\documentclass[12pt]{amsart}
\usepackage{setspace}
\usepackage{verbatim}

\usepackage{amsfonts,fancyhdr,ifthen,bm}
\usepackage{amscd,amssymb} 
\usepackage[letterpaper,left=1.35in,right=1.35in,top=1.35in,bottom=1.35in]{geometry}
 
\usepackage{times}
\usepackage{graphicx}
\usepackage{color}
\usepackage{epsfig}
\usepackage{mathrsfs}
\usepackage{paralist}
\usepackage{comment}
\usepackage{mathtools}
\usepackage{multirow}
\usepackage{tikz-cd}
\usepackage[foot]{amsaddr}
\usepackage{caption}

\usepackage{array}

\usepackage{enumitem}

\usepackage{appendix}

\usepackage{tikz,ifthen,calc}
\usetikzlibrary{automata,arrows,positioning,calc}
 \usetikzlibrary{shapes}

\newtheorem{thm}{Theorem}[section]

\newtheorem*{thm*}{Theorem}
\newtheorem{prop}[thm]{Proposition}
\newtheorem*{prop*}{Proposition}
\newtheorem{cor}[thm]{Corollary}
\newtheorem*{cor*}{Corollary}

\newtheorem{lem}[thm]{Lemma}
\newtheorem*{lem*}{Lemma}

\newtheorem*{oquest*}{Open Question}

\newtheorem*{gconj*}{Goldfeld's Conjecture}
\newtheorem*{bsdconj*}{The Birch and Swinnerton--Dyer  (BSD) Conjecture}

\theoremstyle{remark}
\newtheorem{rmk}[thm]{Remark}

\theoremstyle{remark}
\newtheorem*{rmk*}{Remark}

\theoremstyle{definition}
\newtheorem{defn}[thm]{Definition}

\theoremstyle{definition}
\newtheorem{notat}[thm]{Notation}

\theoremstyle{definition}

\theoremstyle{definition}

\theoremstyle{definition}
\newtheorem*{defn*}{Definition}

\theoremstyle{definition}
\newtheorem{ex}[thm]{Example}

\theoremstyle{definition}

\theoremstyle{definition}

\numberwithin{equation}{section}

 \usepackage{chngcntr}
\counterwithin{figure}{section}

\newcommand{\Z}{\mathbb{Z}}
\newcommand{\QQ}{\mathbb{Q}}

\DeclareMathOperator{\FFF}{\mathbb{F}}

\newcommand{\Sel}{\textup{Sel}}
\newcommand{\Gal}{\textup{Gal}}

\newcommand{\Vplac}{\mathscr{V}}
\newcommand{\Frob}{\textup{Frob}}

\newcommand{\Hom}{\textup{Hom}}

\newcommand{\isoarrow}{\xrightarrow{\,\,\,\sim\,\,\,}}

\newcommand{\res}{\textup{res}}

\newcommand{\CTP}{\textup{CTP}}

\newcommand{\saaa}{s_{\textbf{a}}}
\newcommand{\sbbb}{s_{\textbf{b}}}

\newcommand{\inv}{\textup{inv}}

\DeclareFontFamily{U}{wncy}{}
\DeclareFontShape{U}{wncy}{m}{n}{<->wncyr10}{}
\DeclareSymbolFont{mcy}{U}{wncy}{m}{n}
\DeclareMathSymbol{\Sha}{\mathord}{mcy}{"58}

\newcommand\restr[2]{{
  \left.\kern-\nulldelimiterspace 
  #1 
  \vphantom{\big|} 
  \right|_{#2} 
  }}

\usepackage{scalerel}

\newcommand{\ovQQ}{\overline{\mathbb{Q}}}

\newcommand{\class}[1]{\left[#1\right]}

\newcommand{\Inj}{\textup{Inj}}

\newcommand{\ct}{\textup{ct}}

\renewcommand{\bibliofont}

\begin{document}
\title[BSD implies Goldfeld's Conjecture]{The Birch and Swinnerton--Dyer conjecture implies Goldfeld's conjecture}

\author{Alexander Smith}
\email{asmith13@math.ucla.edu}
\date{\today}

\maketitle

\begin{abstract}
Given an elliptic curve $E/\QQ$, we show that $50\%$ of the quadratic twists of $E$ have $2^{\infty}$-Selmer corank $0$ and $50\%$ have $2^{\infty}$-Selmer corank $1$. As one consequence, we prove that the Birch and Swinnerton--Dyer conjecture implies Goldfeld's conjecture. Previously, this result was known by work of the author for elliptic curves over $\QQ$ satisfying certain technical conditions.

As part of this work, we determine the distribution of $2$-Selmer ranks in the quadratic twist family of $E$. In the cases where this distribution was not already known, it is distinct from the model for distributions of $2$-Selmer groups constructed by Poonen and Rains.
\end{abstract}

\section{Introduction}

Given an elliptic curve $E/\QQ$ with Weierstrass form $y^2 = x^3 + ax + b$ and any $d \in \QQ^{\times}$, define $E^d/\QQ$ to be the elliptic curve with Weierstrass form
\[y^2 = x^3 + d^2ax + d^3b.\]
This is the quadratic twist of $E$ corresponding to $\QQ(\sqrt{d})$.

In 1979, Goldfeld made an influential conjecture about the distribution of analytic ranks in the quadratic twist family of a rational elliptic curve.
\begin{gconj*}[{\cite[Conjecture B]{Gold79}}]
Given an elliptic curve $E/\QQ$ and a nonnegative integer $r$, we have
\[\lim_{H \to \infty} \frac{\#\left\{d \in \Z^{\ne 0} \,:\,\, |d| \le H \,\text{ and }\, r_{\textup{an}}(E^d) = r\right\}}{2H} = \begin{cases} 1/2 &\text{ if }\, r =0\, \text{ or }\, r = 1\\ 0 &\text{ if }\, r\ge 2.\end{cases}\]
\end{gconj*}
Here, the analytic rank $r_{\textup{an}}(E)$ of an elliptic curve $E/\QQ$ is defined as the order of vanishing of the $L$-function corresponding to $E$ at $s = 1$; the analytic rank is always well defined since elliptic curves over $\QQ$ are known to be modular \cite{Breu01}.

This paper proves an analogue of Goldfeld's conjecture for the $2^{\infty}$-Selmer corank, which we will denote by $r_{2^{\infty}}$ and define more precisely below.
\begin{thm}
\label{thm:main}
Given an elliptic curve $E/\QQ$ and a nonnegative integer $r$, we have
\[\lim_{H \to \infty} \frac{\#\left\{d \in \Z^{\ne 0} \,:\,\, |d| \le H \,\text{ and }\, r_{2^{\infty}}(E^d) = r\right\}}{2H} = \begin{cases} 1/2 &\text{ if }\, r =0\, \text{ or }\, r = 1\\ 0 &\text{ if }\, r\ge 2.\end{cases}\]
\end{thm}
Taking $r_{\text{MW}}(E^d)$ to be the Mordell--Weil rank of $E^d/\QQ$, the Birch and Swinnerton--Dyer (BSD) conjecture for $E^d/\QQ$ predicts that
\[r_{\text{MW}}(E^d) = r_{\text{an}}(E^d).\]
Unconditionally, we always have the inequality
\begin{equation}
\label{eq:MW2inf}
r_{\text{MW}}(E^d) \le r_{2^{\infty}}(E^d);
\end{equation}
the Shafarevich--Tate conjecture would imply these ranks are equal. From \cite{Monsky96}, we also have
\begin{equation}
\label{eq:r2ran}
(-1)^{r_{2^{\infty}}(E^d)} = (-1)^{r_{\text{an}}(E^d)} = w(E^d),
\end{equation}
where $w(E^d)$ denotes the global root number of $E^d/\QQ$.

The following result is then an easy consequence of Theorem \ref{thm:main}.
\begin{cor}
\label{cor:BSD2G}
Given any elliptic curve $E/\QQ$, if the BSD conjecture is true in the quadratic twist family of $E$, then Goldfeld's conjecture holds for $E$.
\end{cor}
\begin{proof}
From \eqref{eq:MW2inf} and the assumption of BSD, we have $ r_{\text{an}}(E^d) \le r_{2^{\infty}}(E^d)$. Since $r_{\text{an}}(E^d)$ and $r_{2^{\infty}}(E^d)$ have the same parity, we must have
\[r_{2^{\infty}}(E^d) \le 1 \implies r_{2^{\infty}}(E^d) = r_{\text{an}}(E^d),\]
and the result follows from Theorem \ref{thm:main}.
\end{proof}

Another easy consequence of Theorem \ref{thm:main} and \eqref{eq:r2ran} is the following.
\begin{cor}
\label{cor:unconditional}
Choose any elliptic curve $E/\QQ$. Then, among the quadratic twists $E^d$ with $w(E^d) = +1$, $100\%$ have rank $0$.

Further, among the quadratic twists with $w(E^d) = -1$, $100\%$ have rank at most $1$.
\end{cor}

\begin{ex}
Take $E/\QQ$ to be the elliptic modular curve $X_0(15)$. In \cite{CaNe23}, Caraiani and Newton show that, for a given positive integer $d$, every elliptic curve over $\QQ(\sqrt{-d})$ is modular if $E^{-d}$ has rank $0$. If $d$ is a positive squarefree integer equal to $0$, $1$, $2$, $3$, $4$, $5$, $8$, or $12$ mod $15$, $E^{-d}$ has global root number $+1$. So one consequence of Corolary \ref{cor:unconditional} is that $100\%$ of positive squarefree $d$ in these congruence classes satisfy the condition of Caraiani and Newton. 
\end{ex}

\begin{rmk}
As recounted in \cite[Section 1]{Buru20}, it is known from the work of many authors that
\[r_{\text{an}}(E) = r \implies  r_{p^{\infty}}(E) = r_{\text{MW}}(E) = r\]
for any prime $p$, for $r$ equal to $0$ or $1$, and for any elliptic curve $E/\QQ$. Recently, there have been many successes proving \emph{converse theorems}, which take the form
\[r_{p^{\infty}}(E) = r \implies r_{\text{an}}(E)  = r_{\text{MW}}(E) = r\]
subject to some conditions on $p$, $E$, and $r \le 1$. An archetypal such result is due to Skinner and Urban \cite{SkUr14}.

Relatively few converse theorems have been proved for the case $p =2$, and those that have been proved require $E$ to be a CM elliptic curve \cite[Section A]{ABS22}. Nevertheless, there is reason to be optimistic that Goldfeld's conjecture will eventually be known unconditionally.
\end{rmk}

\subsection{Balanced isogenies and the distribution of $2$-Selmer groups}
For most elliptic curves over $\QQ$, Theorem \ref{thm:main} was established by \cite[Theorem 1.2]{Smi22a}. For that theorem, we needed to assume that $E$ did not have a \emph{balanced isogeny}, a term we now define:
\begin{defn}
A degree $2$ $\QQ$-isogeny $\varphi: E \to E_0$ of elliptic curves will be called a \emph{balanced isogeny} if it satisfies
\[\QQ(E[2]) = \QQ(E_0[2]).\]
Here, $\QQ(E[2])$ denotes the minimal extension of $\QQ$ over which every point in $E[2]$ is rational.
\end{defn}
With this defined, it will be convenient to split the elliptic curves $E/\QQ$ into five cases as follows:
\begin{itemize}[itemsep = 6pt, topsep = 4pt]
\item[] \textbf{Case I}: Either $E(\QQ)[2] = 0$, or $E(\QQ)[2] \cong (\Z/2\Z)^2$ and $E$ has no balanced isogeny.
\item[] \textbf{Case II}: $E(\QQ)[2] \cong \Z/2\Z$ and, taking $\varphi: E \to E_0$ to be the unique degree $2$ $\QQ$-isogeny, $\varphi$ is not balanced and $E_0(\QQ)[2]$ is not $(\Z/2\Z)^2$.
\item[] \textbf{Case III}: $E(\QQ)[2] \cong \Z/2\Z$ and, taking $\varphi: E \to E_0$ to be the unique degree $2$ $\QQ$-isogeny, $E_0(\QQ)[2] \cong (\Z/2\Z)^2$.
\item[] \textbf{Case IV}: There is a unique balanced isogeny $\varphi: E \to E_0$ defined on $E$.
\item[]\textbf{Case V}: There are two distinct balanced isogenies defined on $E$.
\end{itemize}

If $E$ is in Case I or Case II, Theorem \ref{thm:main} follows for $E$ from \cite[Theorem 1.2]{Smi22a}. The main goal of this paper is to establish this theorem for Cases IV and V.
\begin{thm}
\label{thm:main_IVV}
Suppose $E$ has a balanced isogeny. Then Theorem \ref{thm:main} holds for $E$.
\end{thm}
This result implies that Theorem \ref{thm:main} holds for curves in Case III, the one remaining case. After all, if $E$ is in Case III, there is a $\QQ$ isogeny $\varphi: E \to E_0$ to a curve $E_0$ in either Case I, IV, or V. Since we have
\begin{equation}
\label{eq:iso_trick_one}
r_{2^{\infty}}(E^d) = r_{2^{\infty}}(E_0^d)
\end{equation}
for all nonzero integers $d$, Theorem \ref{thm:main} will follow for $E$ since it holds for $E_0$ by either Theorem \ref{thm:main_IVV} or \cite[Theorem 1.2]{Smi22a}.

Balanced isogenies are relevant to us because their presence affects the distribution of $2$-Selmer groups in the quadratic twist family of $E$. To explain this, we first set some notation for Selmer groups and the probabilities we need.

\begin{notat}
\label{notat:early_Tama}
Take $E/\QQ$ to be an elliptic curve and $d$ to be a nonzero integer. Given a positive integer $k$, take $r_{2^k}(E^d)$ to be the $2^k$-Selmer rank of $E^d$; our convention is to define this as the maximal $r$ so that there is some injection
\[(\Z/2^k\Z)^r \hookrightarrow \Sel^{2^{\infty}} E^d \subset H^1\left(G_{\QQ},\, E[2^{\infty}]\right),\]
where $G_{\QQ}$ denotes the absolute Galois group of $\QQ$. The $2^{\infty}$-Selmer corank is then defined by $r_{2^{\infty}}(E^d) = \lim_{k \to \infty} r_{2^k}(E^d)$.

Given a $\QQ$-isogeny $\varphi: E \to E_0$ of degree $2$, there is an associated isogeny from $E^d$ to $E_0^d$ that we also denote by $\varphi$, and we denote the kernel of this map by $E^d[\varphi]  \cong E[\varphi]$. We then take
\[r_{\varphi}(E^d) = \dim_{\FFF_2} \text{im}\left(H^1(G_{\QQ}, E[\varphi]) \to H^1(G_{\QQ}, E^d[2^{\infty}])\right) \cap \Sel^{2^{\infty}} E^d.\]

Given $H > 0$ and a nonzero integer $d_0$, take $X_E(d_0, H)$ to be the set of squarefree integers $d$ of magnitude at most $H$ such that $dd_0$ lies in $\left(\QQ_v^{\times}\right)^2$ for $v$ equal to $2$, $\infty$, and all places of bad reduction for $E$. If $H$ is sufficiently large, we may choose $d$ in $X_E(d_0, H)$ such that
\[E^d(\QQ)[2] = E^d(\QQ)[4] \quad\text{and}\quad E_0^d(\QQ)[2] = E_0^d(\QQ)[4].\]
Taking $\varphi': E_0 \to E$ to be the dual isogeny to $\varphi$, we then define
\[u(\varphi, d_0) = r_{\varphi}(E^d) - r_{\varphi'}(E_0^d).\]
As we will see in Definition \ref{defn:Tama}, this does not depend on the choice of $d$.
\end{notat}

\begin{notat}
Given integers $n \ge j \ge 0$ and $m \ge 0$, we take $P^{\text{Mat}}(j \,|\, m \times n)$ to be the probability that a uniformly selected random $m \times n$ matrix with coefficients in $\FFF_2$ has kernel of dimension $j$. Given any integer $u$, we then take
\[P^{\text{Mat}}(j\,|\, (\infty - u) \times \infty)= \lim_{n \to \infty} P^{\text{Mat}}(j \,|\, (n- u) \times n).\]
We also take $P^{\text{Alt}}(j\,|\,n)$ to be the probability that a uniformly selected random alternating $n \times n$ matrix with coefficients in $\FFF_2$ has kernel of rank $j$.

An explicit formula for these probabilities appears in \cite[Case 2.12]{Smi22b}.
\end{notat}

\begin{thm}
\label{thm:balanced}
Choose a balanced isogeny $\varphi: E \to E_0$ and a nonzero integer $d_0$, and take $u = u(\varphi, d_0)$. Then, for any $r \ge 0$,
\[\lim_{H \to \infty} \frac{\# \{d \in X_E(d_0, H)\,:\,\, r_{\varphi}(E^d) = r\}}{ \#  X_E(d_0, H)} = P^{\textup{Mat}}(r\,|\, (\infty - u) \times \infty).\]
\end{thm}

\begin{rmk}
Choose an elliptic curve $E/\QQ$ in Case II, take $\varphi: E \to E_0$ to be the unique $\QQ$-isogeny of degree $2$, and take $\varphi': E_0 \to E$ to be the dual isogeny to $\varphi$. In \cite{KaKl17}, Kane and Klagsbrun showed that
\[\lim_{H  \to \infty} \frac{\#\left\{d :\,\, |d| \le H, \,\, r_{\varphi}(E^d) = r, \,\text{ and }\, r_{\varphi'}(E_0^d) = r- u\right\}}{\#\left\{d\,:\,\, |d| \le H \, \text{ and }\ r_{\varphi}(E^d) - r_{\varphi'}(E_0^d) = u\right\}} = P^{\textup{Mat}}(r\,|\, (\infty - u) \times \infty)\]
for any integers $u$ and $r \ge \max(u, 0)$. Theorem \ref{thm:balanced} may be thought of as an extension of this result.

Without the restriction to twists satisfying $r_{\varphi}(E^d) - r_{\varphi'}(E_0^d) = u$, the distribution of $\varphi$-Selmer ranks in the quadratic twist family of $E$ for a non-balanced isogeny $\varphi: E \to E_0$ is degenerate. Specifically,
\begin{itemize}
\item If $E$ is in Case II, then $r_{\varphi}(E^d)$ blows up for $50\%$ of $d$ (that is, is greater than $r$ for any positive integer $r$ for at least $50\%$ of $d$) and is $0$ for $50\%$ of $d$.
\item If $E(\QQ)[2] \cong (\Z/2\Z)^2$, then $r_{\varphi}(E^d)$ is $0$ for $100\%$ of $d$.
\item If $E$ is in Case III, then $r_{\varphi}(E^d)$ blows up for $100\%$ of $d$.
\end{itemize}
\end{rmk}

For curves in Case IV and V, the distribution of $2$-Selmer ranks is built up from the distribution of $\varphi$-Selmer ranks in the family. For curves in Case IV, this takes the following form:
\begin{thm}
\label{thm:2IV}
Take $E/\QQ$ to be a curve in Case IV, and take $\varphi: E \to  E_0$ to be its unique balanced isogeny. Choose a nonzero integer $d_0$, and take $u = u(\varphi, d_0)$. Then, given any integers $r_2 \ge r_{\varphi} \ge \max(u, 0)$, we have
\[\lim_{H \to \infty} \frac{\# \left\{d \in X_E(d_0, H)\,:\,\, r_{\varphi}(E^d) = r_{\varphi}\, \text{ and }\, r_2(E^d) = r_2\right\}}{\# \left\{d \in X_E(d_0, H)\,:\,\, r_{\varphi}(E^d) = r_{\varphi}\right\}} = P^{\textup{Alt}}\left(r_2 - r_{\varphi}\,|\, r_{\varphi} - u\right).\]
\end{thm}

For curves in Case V, the distribution is further complicated by the presence of a second balanced isogeny. We require some more notation.
\begin{notat}
Choose nonnegative integers $j$, $n$, and $m$, with $m$ even. Fix a nondegenerate alternating pairing $P_0$ on $\FFF_2^m$. Take $T: \FFF_2^n \to \FFF_2^m$ to be a random homomorphism chosen uniformly among all the possibilities, and define a random alternating pairing $P$ on $\FFF_2^n$ by $P(v, w) = P_0(Tv, Tw)$. We then define
\[P^{\text{V}}(j\,|\, n \to m)\]
to be the probability that $P$ has kernel of rank exactly $j$.
\end{notat}

\begin{thm}
\label{thm:2V}
Choose an elliptic curve $E/\QQ$ in Case V, choose distinct balanced isogenies $\varphi_1: E \to E_1$, $\varphi_2: E \to E_2$, and choose a nonzero integer $d_0$. For $i = 1,2$, take $u_i = u(\varphi_i, d_0)$, and take $u_0 = -u_1 - u_2$. Then $u_0$ is a nonnegative even integer. Further, given nonnegative integers $r_2 \ge r_{\varphi_1} \ge \max(0, u_1)$, we have
\begin{align*}
&\lim_{H \to \infty} \frac{\# \left\{d \in X_E(d_0, H)\,:\,\, r_{\varphi_1}(E^d) = r_{\varphi_1}\, \text{ and }\, r_2(E^d) =r_2\right\}}{\# \left\{d \in X_E(d_0, H)\,:\,\, r_{\varphi_1}(E^d) = r_{\varphi_1}\right\}}\\
& \qquad= P^{\textup{V}}\left(r_2 - r_{\varphi_1}\,|\, (r_{\varphi_1} - u_1) \to u_0\right).
\end{align*}
\end{thm}

\begin{rmk}
Given an elliptic curve $E/\QQ$, we may choose a finite list of nonzero integers $d_{01}, \dots, d_{0k}$ such that, for any $H > 0$,
\[X_E(d_{01}, H), \dots, X_E(d_{0k}, H)\]
are disjoint sets whose union is the set of all squarefree integers of magnitude at most $H$. If $E$ is in Case IV or V, we can use this decomposition together with Theorems \ref{thm:balanced}, \ref{thm:2IV}, and \ref{thm:2V} to determine the distribution of $r_2(E^d)$ as $d$ varies among all squarefree integers.

Unlike for $E$ in Cases II and III, these distributions are nice, and $2^{kr_2(E^d)}$ has finite average over these families for any $k > 0$ \cite[Theorem 2.6]{Smi22b}. However, the distribution of Selmer ranks is different from the distribution of $2$-Selmer ranks constructed by Delaunay and Poonen--Rains \cite{Del01,PR12}, which is how the $2$-Selmer ranks are distributed in the quadratic twist family of a curve in Case I \cite[Theorem 1.5]{Smi22a}.

This is easiest to verify by considering the probability that the $2$-Selmer rank of $E^d$ is greater than some large integer $r$ as $d$ varies. For a curve in Case I, this probability is on the order of 
\[2^{-\frac{1}{2}r^2 + O(r)}.\]
For a curve in Case IV, the probability is instead on the order of
\[2^{-\frac{3}{8}r^2  + O(r)},\]
and for a curve in Case V it is on the order of
\[2^{-\frac{1}{4}r^2 + O(r)}.\]
Here, the implicit constants depend on the curve $E$.
\end{rmk}

\subsection{Isogeny tricks}
\label{ssec:isogeny_tricks}
In \cite{Smi22a}, we gave a method for finding the distribution of $2^k$-Selmer ranks in certain quadratic twist families of Galois modules where the $2$-Selmer structure was fixed. This method made a number of assumptions on the underlying $2$-Selmer structure, and the resulting theorem was all-or-nothing; the paper made no partial claims in the case where the assumptions were not satisfied.

We begin our work by considering what the method of \cite{Smi22a} gives in the case where some of the technical conditions are not satisfied. Our result is Theorem \ref{thm:controllable}; with an eye to future applications, this result is stated in greater generality than is needed for the rest of the paper.

When we apply this theorem to curves in Case IV and V, we do not immediately get Theorem \ref{thm:main_IVV}. The result instead breaks into cases based on a parameter we now define.
\begin{defn}
Choose a $\QQ$-isogeny $\varphi: E \to E_0$ of degree $2$ and a nonzero integer $d$. We take $\Sel^{2^{\infty}}_{\text{div}} E^d$ to be the subgroup of divisible elements in the $2^{\infty}$-Selmer group of $E^d$, and we define
\[r_{\varphi,\, \textup{div}}(E^d) = \dim_{\FFF_2} \text{im}\left(H^1(G_{\QQ}, E[\varphi]) \to H^1(G_{\QQ}, E^d[2^{\infty}])\right) \cap \Sel^{2^{\infty}}_{\text{div}} E^d.\]
\end{defn}

\begin{thm}
\label{thm:IVV_pretrick}
Suppose that $E$ is an elliptic curve in Case IV, and take $\varphi$ to be the unique balanced isogeny from $E$. Then, for $100\%$ of squarefree integers $d$, either
\begin{align*}
&r_{2^{\infty}}(E^d) \le 1 \qquad \text{or}\\
&r_{2^{\infty}}(E^d) = r_{\varphi,\, \textup{div}}(E^d) \ge 2.
\end{align*}

Supposing instead that $E$ is in Case V, take $\varphi_1$ and $\varphi_2$ to be distinct balanced isogenies defined on $E$. Then, for $100\%$ of squarefree integers $d$, either
\begin{alignat*}{2}
&r_{2^{\infty}}(E^d) \le 1\qquad&&\text{or}\\
&r_{2^{\infty}}(E^d) = r_{\varphi_1,\, \textup{div}}(E^d) \ge 2\qquad&&\text{or}\\
&r_{2^{\infty}}(E^d) = r_{\varphi_2,\, \textup{div}}(E^d) \ge 2. &&
\end{alignat*}
\end{thm}
The key insight is that, by applying this theorem not just to $E$ but also to the other curves in Case IV and V that are isogenous to $E$, we are able to prove that the first possibilities above are the only ones that happen with nonzero probability. This starts with the following proposition.
\begin{prop}
\label{prop:trick}
Choose a $\QQ$-isogeny $\varphi: E \to E_0$ of degree $2$ and a nonzero integer $d$, and take $\varphi': E_0 \to E$ to be the dual isogeny to $\varphi$. Then
\[r_{2^{\infty}}(E^d) = r_{2^{\infty}}(E_0^d) = r_{\varphi,\textup{div}}(E^d) + r_{\varphi',\textup{div}}(E_0^d).\]
\end{prop}
\begin{proof}
We may decompose the multiplication by $2$ map on $\Sel_{\text{div}}^{2^{\infty}}E^d$ as the composition of maps
\begin{equation}
\label{eq:div_decomp}
\Sel_{\text{div}}^{2^{\infty}}E^d \xrightarrow{\,\, \varphi\,\,} \Sel_{\text{div}}^{2^{\infty}}E_0^d \xrightarrow{\,\,\varphi'\,\,} \Sel_{\text{div}}^{2^{\infty}}E^d.
\end{equation}
From the exact sequence
\[H^1(G_{\QQ}, E^d[\varphi]) \to H^1(G_{\QQ}, E^d[2^{\infty}]) \to H^1(G_{\QQ}, E^d_0[2^{\infty}]),\]
we find that the first map in \eqref{eq:div_decomp} has kernel of dimension $r_{\varphi,\, \textup{div}}(E^d)$; similarly, the second map has kernel of dimension $r_{\varphi', \, \textup{div}}(E_0^d)$. Meanwhile, their composition has kernel of dimension $r_{2^{\infty}}(E^d)$, so $r_{2^{\infty}}(E^d) = r_{\varphi,\textup{div}}(E^d) + r_{\varphi',\textup{div}}(E_0^d)$.

Since the first map in \eqref{eq:div_decomp} has finite kernel and $r_{2^{\infty}}(E^d)$ equals the corank of $\Sel^{2^{\infty}}E^d$, we have $r_{2^{\infty}}(E^d) = r_{2^{\infty}}(E_0^d)$.
\end{proof}

We can now prove Theorem \ref{thm:main_IVV} and, with it, Theorem \ref{thm:main}.
\begin{proof}[Proof of Theorem \ref{thm:main_IVV}]
Suppose Theorem \ref{thm:main_IVV} did not hold for a curve $E$ in Case IV or V, so that a positive proportion of quadratic twists of $E$ had $2^{\infty}$-Selmer corank greater than $1$. By considering the isogenous curve and applying Proposition \ref{prop:trick} if necessary, we may assume that $E$ is either in Case V or is not isogenous to a curve in Case V.

By Theorem \ref{thm:IVV_pretrick}, there is a balanced isogeny $\varphi: E \to E_0$ such that
\[r_{2^{\infty}}(E^d) = r_{\varphi,\, \textup{div}}(E^d) \ge 2.\]
occurs for a positive proportion of $d$. Take $\varphi': E_0 \to E$ to be the dual isogeny to $\varphi$; it is also a balanced isogeny. By assumption if $E$ is in Case IV and by \cite{Chil21} if $E$ is in Case V, we find that $E_0$ is in Case IV.

Proposition \ref{prop:trick} then gives that
\[r_{2^{\infty}}(E_0^d) - r_{\varphi',\,\text{div}}(E_0^d) \ge 2\]
for a positive proportion of $d$. But this is inconsistent with Theorem \ref{thm:IVV_pretrick} applied to $E_0$, and we have a contradiction.
\end{proof}

\begin{rmk}
Given a number field $F$, the above argument can be extended without much extra work to an elliptic curve $E$ over $F$ if
\begin{itemize}
\item The equation $x^2 + y^2 = -1$ has no solution over $F$ and
\item There is no imaginary quadratic field $K$ contained in $F$.
\end{itemize}
Given a field $F$ obeying these two conditions, if we order the $d$ in $F^{\times}/(F^{\times})^2$ by the minimal magnitude of the norm of an integral representative of $d$, we can show that $r_{2^{\infty}}(E^d)$ is less than two for $100\%$ of $d$. The actual distribution $2^{\infty}$-Selmer ranks then depends on the disparity of the curve over $F$ \cite{KMR13}.

More specifically, the problematic cases are first when $E$ has CM by the order of a quadratic field contained in $F$, and second when $E[2]$ has an automorphism that commutes with the connecting maps in the sense of \cite[Section 2]{Smi22b}. From \cite[Section 3.1]{Smi22b}, the latter case only can occur if $x^2 + y^2 = -1$ is solvable in $F$.

For these cases, isogeny tricks do not suffice to solve the problem, and it is unclear how to proceed.
\end{rmk}

\subsection{Outline of the paper}
In Section \ref{sec:general_Smi}, we consider the machinery for controlling the distribution of higher Selmer groups given in \cite{Smi22a} and modify it so it gives partial answers for a wider range of questions.

In Section \ref{sec:IVV}, we assemble the necessary background on curves in Cases IV and V to streamline the process of determining the distribution of $2$-Selmer ranks in their quadratic twist  families. With this done, we apply the methods of \cite{Smi22b} to determine refined moments of the $2$-Selmer groups in these families in Section \ref{sec:grids}, and then reconstruct the distribution of ranks from these moments in Section \ref{sec:ranks}. Finally, in Section \ref{sec:2infty}, we prove Theorem \ref{thm:IVV_pretrick} by applying the work of Section \ref{sec:general_Smi}.

\subsection*{Acknowledgements}
This paper was motivated by a question from Ana Caraiani and James Newton, who asked if the methods in \cite{Smi22a} could be applied to the Case V elliptic curve $X_0(15)$. The author is grateful for this question.

The author served as as Clay Research Fellow during the writing of this paper, and would like to thank the Clay Mathematics Institute for their support.

\section{General results on the distribution of higher Selmer ranks}
\label{sec:general_Smi}
The work on the distribution of higher Selmer groups in \cite{Smi22a} is stated at a couple of different levels of generality. First, there is \cite[Theorem 4.18]{Smi22a}, which gives the distribution of higher Selmer groups in certain subsets of grids of twists subject to a number of technical conditions. This is proved as a consequence of \cite[Theorem 6.2]{Smi22a}, which is an equidistribution statement for the Cassels--Tate pairing over certain finer subsets of the grid of twists, and which is proved subject to the same technical conditions as the former theorem.

Our goal for this section is to prove \cite[Theorem 6.2]{Smi22a} with some of the technical conditions weakened. The original proof of this result is long and difficult. Fortunately, there are relatively few places where the technical conditions we are modifying are used in the original proof. So, to avoid redundancy, we shall prove our generalized form of this result by enumerating where the skipped conditions appear in the proof and giving alternative arguments for these affected steps.

In contrast to the rest of this paper, in this section we will allow modules over number fields other than $\QQ$, and we will allow twist families of degree other $2$.

\subsection{Some setup from \cite{Smi22a}}
\label{ssec:CTP}

To start, fix a rational prime $\ell$ and a positive integer $k_0$. Take $\xi$ to be the image of $x$ in the quotient ring
\[\Z_{\ell}[\xi] = \Z_{\ell}[x]\Big/\left(1 + x^{L} + x^{2 L} + \dots + x^{(\ell - 1) L}\right)\quad\text{with}\quad L = \ell^{k_0 - 1}\]
and take $\omega = \xi - 1$.

 Fix a number field $F$, and take $N$ to be a $\Z_{\ell}[\xi][G_F]$ module isomorphic to some positive power of $\QQ_{\ell}/\Z_{\ell}$ as a topological module, where $G_F$ denotes $\Gal(\overline{F}/F)$ for some algebraic closure $\overline{F}$ of $F$. We assume that the action of $G_F$ on $N$ is ramified at finitely many places. We refer to $N$ as a twistable module.

Choose a Galois field extension $K/F$ contained in $\overline{F}$ and a finite set of places $\Vplac_0$ of $F$ so that $(K/F, \Vplac_0)$ unpacks $N$, in the sense of \cite[Definition 4.2]{Smi22b}. Choose a set of local conditions $(W_v(\chi))_{v \in \Vplac_0, \chi}$ for this module, and use it to define the Selmer groups of $N$ in its twist family \cite[Definition 4.3]{Smi22a}. These Selmer groups are $\Z_{\ell}[\xi]$ modules.

We will assume that $N$ has alternating structure \cite[Definnition 4.12]{Smi22a}, so that there are compatible nondegenerate alternating $\Z_{\ell}[\xi][G_F]$-equivariant pairings
\[N[\ell^k] \otimes N[\ell^k] \to \mu_{\ell^k} \quad\text{for }\, k \ge 1\]
such that the preimage of $W_v(\chi)$ in $H^1(G_v, N^{\chi}[\ell^k])$ is its own orthogonal complement with respect to the corresponding local Tate pairing for all $v \in \Vplac_0$ and all twists $\chi$.

There is then an associated antisymmetric Cassels--Tate pairing
\[\langle\,\,,\,\, \rangle: \Sel\, N^{\chi} \otimes \Sel\, N^{\chi} \to \QQ_{\ell}/\Z_{\ell}\]
for every continuous homomorphism $\chi: G_F \to \langle \xi \rangle$ \cite[Definition 4.11]{Smi22a}. The kernel of this pairing is the subgroup of divisible elements in $\Sel\, N^{\chi}$.
This pairing satisfies
\begin{equation}
\label{eq:unitary}
\langle \alpha \phi, \, \psi \rangle = \langle \phi,\, \overline{\alpha} \psi \rangle\,\,\text{ for all }\, \phi, \psi \in \Sel\,N^{\chi}\, \text{ and }\, \alpha \in  \Z_{\ell}[\xi],
\end{equation}
where $\overline{\alpha}$ is the image of $\alpha$ under the continuous ring automorphism of $\Z_{\ell}[\xi]$ taking $\xi$ to $\xi^{-1}$. 

Now choose any integer $k \ge 1$, and take $A = \Sel \,N^{\chi}$. Applying \eqref{eq:unitary}, we see that there is a well-defined pairing
\[\langle\,\,,\,\, \rangle_k: \omega^{k-1}\left (A[\omega^k]\right) \otimes \omega^{k-1} \left(A[\omega^k]\right)  \to \tfrac{1}{\ell}\Z/\Z\]
given by
\[\langle \omega^{k-1} \alpha, \,\omega^{k-1} \beta\rangle_k = \langle\alpha, \,\omega^{k-1} \beta\rangle\quad\text{for all}\quad \alpha, \beta \in A[\omega^k].\]
This pairing is antisymmetric for $k$ odd and symmetric for $k$ even, and its kernel is  $\omega^k A[\omega^{k+1}]$.

\begin{rmk}
This last construction is a slight variant of \cite[Definition 4.11]{Smi22a}. For $k \ge 1$, the long exact sequence associated to $0 \to N^{\chi}[\omega] \to N^{\chi} \to N^{\chi} \to 0$ yields an exact sequence
\begin{equation}
\label{eq:H0_omega_torsion}
0 \to \frac{H^0(G_F, N^{\chi})}{\omega H^0(G_F, N^{\chi})} \xrightarrow{\quad} \omega^{k-1} \Sel^{\omega^k} N^{\chi} \xrightarrow{\quad} \omega^{k-1} \left(\Sel\, N^{\chi}\right)\left[\omega^k\right]  \to 0.
\end{equation}
In this way, the pairing just defined extends to a pairing on $\omega^{k-1} \Sel^{\omega^k} N^{\chi}$, which is the form of the pairing considered in \cite{Smi22a}. 
\end{rmk}

\subsection{Grid classes}
\label{ssec:grid_class}
We start by summarizing some of \cite[Definition 4.7]{Smi22a}, which considers \emph{grids of twists} of the form
\[\chi: X \to \Hom(G_F, \langle \xi \rangle).\]
 Here, $X$ is a product space $\prod_{s \in S} X_s$, where $S$ is a finite set and the $X_s$ are finite sets of primes of $\overline{F}$. Among other requirements, grids are defined so $\chi((x_s)_{s \in S})$ is ramified at the primes in
\[\{x_s \cap F\,:\,\, s \in S\}\]
and at no other primes outside $\Vplac_0$ for all $(x_s)_{s  \in S}$ in $X$. Choosing some $x_0$ in $X$, we will focus on the \emph{grid class} $[x_0] \subseteq X$ defined with respect to $(K/F, \Vplac_0)$. On these grid classes, we may construct a space $V_{\omega}(x_0)$ and a map
\[\Psi_{x}: V_{\omega}(x_0) \to H^1(G_F, N[\omega])\]
for each $x \in \class{x_0}$ so the image of $\Psi_{x}$ is identified with the $\omega$-Selmer group of $N^{\chi(x)}$ under the twisting map $N[\omega] \isoarrow N^{\chi(x)}[\omega]$.

From \eqref{eq:H0_omega_torsion}, we then get an exact sequence
\[0 \to \frac{H^0(G_F, N^{\chi}(x))}{\omega H^0(G_F, N^{\chi(x)})}  \to V_{\omega}(x_0) \to \left(\Sel \,N^{\chi(x)}\right)[\omega] \to 0.\]
The image of the first term does not depend on the choice of $x$, and we will denote it by $V_{\text{tor}}(x_0)$.

For $s \in S$, we have a map
\begin{equation}
\label{eq:ram_measure}
\pi_s: V_{\omega}(x_0) \to N[\omega](-1).
\end{equation}
This map is defined such that, given any $x \in \class{x_0}$ and $v$ in $V_{\omega}(x_0)$, the value $\pi_s(v_0)$ measures the ramification of $\Psi_{x}(v)$ at $\pi_s(x) \cap F$.

\subsection{Higher grid classes}
Given $x \in \class{x_0}$ and $j \ge 1$, we define $V_{\omega^j}(x)$ to be the preimage of $\omega^{j-1} \left(\Sel\, N^{\chi(x)}\right)[\omega^j]$ in $V_{\omega}(x_0)$. Composing the pairing $\langle\,\,,\,\,\rangle_j$ with $\Psi_{x} \otimes \Psi_{x}$ then allows us to define a pairing
\[\langle\,\,,\,\,\rangle_{j,\, x}\,:\, V_{\omega^j}(x) \otimes V_{\omega^j}(x) \to \tfrac{1}{\ell}\Z/\Z\]
with kernel $V_{\omega^{j+1}}(x)$.

As in \cite[Definition 4.15]{Smi22a}, given an integer $k \ge 1$, we define the \emph{higher grid class} $\class{x_0}_k$ to be the set of $x \in \class{x_0}$ such that
\begin{itemize}
\item The spaces $V_{\omega^j}(x)$ and $V_{\omega^j}(x_0)$  are equal for $j \le k$ and
\item The pairings $\langle\,\,,\,\,\rangle_{j,\, x}$ and  $\langle\,\,,\,\,\rangle_{j, \,x_0}$ are equal for $j < k$.
\end{itemize}

\subsection{Stating the general result}
\label{ssec:general_Smi1}
\begin{defn}
\label{defn:Spre}
In \cite[Definition 4.19]{Smi22a}, we codify the of a given grid class $\class{x_0} \subseteq X$ being \emph{ready for higher work}, which is in turn based on a definition of subsets $S_{\textup{pot-pre}}, S_{\textup{pot-a/b}} \subseteq S$ of \emph{potential prefix indices} and \emph{potential a/b indices}. These definitions involve the \emph{height $H$ of $X$}, which is defined as the maximum value of the norm
\[\prod_{s \in S}N_{F/\QQ} \left( x_s \cap F\right)\]
attained as $(x_s)_s$ varies through $X$.

For this paper, we say that $\class{x_0}$ is \emph{ready for some higher work} if $H > 20$, if conditions (1), (2), and (3) of \cite[Definition 4.19]{Smi22a} are satisfied and if, for all integers $k < \log \log \log H$, there is a subset $S_{\text{pre}} \subseteq S_{\text{pot-pre}}$ of cardinality $k$ and a subspace
\[\text{k}V_{\omega} \subseteq V_{\omega}(x_0)\]
such that the kernel of \eqref{eq:ram_measure} is $\text{k}V_{\omega}$ for all $s \in S_{\text{pre}}$ and such that
\[V_{\omega}(x_0) =\text{k}V_{\omega}  + V_{\text{tor}}(x_0).\]
We then call $S_{\text{pre}}$ a \emph{set of prefix indices}.

We take $r_{\omega^j}(x_0)$ as shorthand for $\dim_{\FFF_{\ell}} \textup{k}V_{\omega}\cap V_{\omega^k}(x_0)$.
\end{defn}

\begin{defn}
\label{defn:half_higher}
Take $X$ to be a grid of height $H$, and take $\class{x_0}$ to be a grid class that is ready for some higher work. Choose a positive integer $k < \log \log \log H$, choose some set $S_{\textup{pre}}$ of prefix indices, and define $\text{k}V_{\omega, N}$ and $\text{k}V_{\omega, N^{\vee}}$ as in Definition \ref{defn:Spre}. Choose
\[w \in \left( \text{k}V_{\omega} \cap V_{\omega^k}(x_0) \right)^{\otimes 2}.\]
Given $x \in \class{x_0}$, we may evaluate $\langle\,\,,\,\,\rangle_{k, \,x}$ at $w$ so long as $V_{\omega^k}(x) \otimes V_{\omega^k}(x)$ contains $w$. In particular, we may do this evaluation throughout the higher grid class $\class{x_0}_k$. We write this evaluation as $\textup{ct}_{x, k}(w)$.

We say that $w$ is \emph{controllable} if there are indices $\saaa, \sbbb \in S_{\textup{pot-a/b}}$ such that we have
\[(\pi_{\saaa} \otimes \pi_{\saaa})(w) = 0 \quad\text{or} \quad (\pi_{\sbbb} \otimes \pi_{\sbbb})(w) = 0\quad\text{in }\, \left(N[\omega](-1)\right)^{\otimes 2}\]
and such that the tensor
\[\left(\pi_{\saaa} \otimes \pi_{\sbbb} + (-1)^{k-1} \pi_{\sbbb} \otimes \pi_{\saaa}\right)(w)\]
is not sent to $0$ under the pairing
\[\left(N[\omega](-1)\right)^{\otimes 2} \to \mu_{\ell}(-2)\]
given in \cite[Definition 4.9]{Smi22a}.
\end{defn}

\begin{thm}
\label{thm:controllable}
Choose $N$, $\xi$, and $(K/F, \Vplac_0)$ as above. Then there are real numbers $c, C > 0$ with $c$ an absolute constant and $C$ depending just on $N$, $\xi$, and $(K/F, \Vplac_0)$ for which the following holds:

Choose a grid class $\class{x_0} \subseteq X$ that is ready for some higher work as in Definition \ref{defn:Spre}, and take $H$ to be the height of $X$. We assume $H > C$.

In addition, choose a positive integer $k \le \log \log \log H$, and choose $S_{\textup{pre}}$ and a controllable tensor $w$ as in Definition \ref{defn:half_higher}.  Suppose we have
\[(r_{\omega}(x_0) + k) \cdot \left(r_{\omega}(x_0) + \dots + r_{\omega^k}(x_0)\right) \le \frac{c  \log \log \log H}{(\dim N[\omega])^2 \cdot \log \ell}.\]
Then, taking $Y = \class{x_0}_k$, we have
\begin{equation}
\label{eq:test_mean}
\left|\sum_{x \in Y} \exp\left(2\pi i \cdot \textup{ct}_{x, k}(w)\right)\right| \le \left(\log \log H\right)^{-1/4} \cdot \#\class{x_0}.
\end{equation}
\end{thm}
\begin{rmk}
We mention that the module $N \oplus N^{\vee}$ has alternating structure even if $N$ does not; here, $N^{\vee}$ denotes the dual twistable module to $N$ \cite[Definition 4.9]{Smi22a}. So the assumption that $N$ has alternating structure in Theorem \ref{thm:controllable} may be circumvented by instead applying this result to $N \oplus N^{\vee}$.
\end{rmk}

\subsection{The proof of Theorem \ref{thm:controllable}}
Choose $\saaa$ and $\sbbb$ from $w$ as in Definition \ref{defn:half_higher}. Without loss of generality, we may assume that $(\pi_{\saaa} \otimes \pi_{\saaa})(w)$ is $0$. Taking $V_0 = \text{k}V_{\omega} \cap V_{\omega^k}(x_0)$ and $V_{0a} = \ker \pi_{\saaa} \cap V_0$, we have
\[w \in \left(V_0 \otimes V_{0a} \right) \oplus \left(V_{0a} \otimes V_0\right).\]
Since
\[\ct_x(v_1 \otimes v_2) = (-1)^k \ct_x(v_1 \otimes v_2) \quad\text{for all } \,v_1, v_2 \in V_0\]
for all $x$ in $Y$ \cite[Remark 4.14]{Smi22a}, we may find a controllable tensor $w_0$ in $V_0 \otimes V_{0a}$ such that $\ct_x(w_0) = \ct_x(w)$  for all $x$ in $Y$. So we may assume $w$ lies in $V_0 \otimes V_{0a}$ without loss of generality. Write $w$ in the form
\[w = \sum_{i = 1}^m v_{i} \otimes v_{i}'\quad\text{with }\, v_i \in V_0, v'_i \in V_{0a}\text{ for } i \le m,\]
where $m$ is at most $\dim V_0$.

With this done, the proof of the result largely follows the proof of \cite[Theorem 6.2]{Smi22a}. We will now go through where a deviation in the proof is needed:

First, rather than selecting $\saaa$ and $\sbbb$ according to the procedure in \cite[Definition 6.3]{Smi22a}, we select them as above. The same arguments as in \cite[Section 6.2]{Smi22a} then allows us to reduce Theorem \ref{thm:controllable} to the analogus statement on a set $\class{x_0} \cap (Z_{\text{pre}} \times X_{\saaa}(Z_{\text{pre}}) \times X_{\sbbb}(Z_{\text{pre}}))$ as defined in \cite[Theorem 6.4]{Smi22a}.

We next modify \cite[Notation 8.15]{Smi22a}. In this step, we had split the primes in $X_{\saaa}(Z_{\text{pre}})$ into equivalence classes based on their governing class defined from a governing expansion corresponding to a single vector $v_a \in V_{\omega^k, N}(x_0)$. Instead of this, first choose $v_{m+1}, \dots, v_{m+c}$ in $\text{k}V_{\omega}$ so that $v_1, \dots, v_{m+c}$ is a basis for $\text{k}V_{\omega}$. We note that $(v_1, \dots, v_{m+c})$ parameterizes an $\omega$-Selmer element of $\left(N^{\oplus m+c}\right)^{\chi(x_0)}$, and we choose a governing expansion corresponding to this tuple for $N^{\oplus m +c}$. Since $m+c$ equals $\dim \text{k}V_{\omega}$, we see that the bound for the number of governing classes given in \cite[Proposition 8.16]{Smi22a} still holds for this alternative setup. We then split $X_{\saaa}(Z_{\text{pre}})$ into classes based on this governing expansion. This is also the governing expansion to which we will apply \cite[Propositions  8.19 and 8.20]{Smi22a}.

We finally need to modify \cite[Theorem 8.12]{Smi22a} to control $\ct_{x, k}(w)$ under the condition that $X_{\mathbf{a}}$ lies in a single equivalence class under the definition of the previous paragraph. Our approach is to attempt to apply the argument of this theorem over $N^{\oplus m}$ to $v_a = (v_1, \dots, v_m)$ and $v_b = (v'_1, \dots, v'_m)$. We check how this proof must be modified.

First, we note that the strict signature at $x$ for $N^{\oplus m}$ is determined from the strict signature at $x$ for $N$, so we do not need to increase the bound on the number of possible strict signatures. The set of points containing a given strict signature is closed in this instance by \cite[Lemma 8.10]{Smi22a} and the fact that we have a governing expansion for every $v \in \text{k}V_{\omega}$, as follows from the assumption that $v_1, \dots, v_{m + c}$ generate this space. By applying \cite[Lemma 8.11]{Smi22a} with $S^{\circ}$ taken to be $S_{\text{pre}}$, we may prove that the set of points with a given strict signature in a given higher class $\class{x_0}_k$ is also closed.

Finally, we still have that the factor $\nu$ appearing in \cite[Theorem 8.12]{Smi22a} is nonzero; this is equivalent to the assumption we have made that
\[\left(\pi_{\saaa}(v_1), \dots, \pi_{\saaa}(v_m)\right), \,\,\left(\pi_{\sbbb}(v_1'), \dots, \pi_{\sbbb}(v_m')\right) \in N[\omega]^{\oplus m}(-1)\]
have nonzero pairing in $\mu_{\ell}(-2)$. The rest of the argument proceeds as in \cite{Smi22a}.
 \qed

\section{Algebraic results for elliptic curves in Cases IV and V}
\label{sec:IVV}

\subsection{Tamagawa ratios}
\label{ssec:Tama}

The Selmer groups associated to a balanced isogeny in a quadratic twist family have a nice, non-degenerate distribution. The reason for this comes down to the theory of their Tamagawa ratios, which we go through now.
\begin{defn}
\label{defn:Tama}
Suppose $\varphi: E \to E_0$ is a $\QQ$-isogeny of degree $2$.  Given a nonzero integer $d$ and a rational place $v$, we take
\[W_v(\varphi, d) = \ker\left(H^1(G_v, E^{d}[\varphi]) \to  H^1(G_v, E^{d})\right).\]
We then define the \emph{Tamagawa ratio} to be the product
\[\mathcal{T}(\varphi, d) = \prod_{v} \frac{\# W_v(\varphi, d)}{2}.\]
If $v$ is not a place of bad reduction for $E^d$ or $2$ or $\infty$, then $W_v(\varphi, d)$ consists of the unramified cocycle classes in $H^1(G_v, E^d[\varphi])$ and so has order $2$. As a result, the product defining the Tamagawa ratio converges. 

We then take $u(\varphi, d)$ to be the integer satisfying
\[\mathcal{T}(\varphi, d) = 2^{u(\varphi, d)}.\]

Define the isogeny Selmer group $\Sel^{\varphi} E^d$ as the preimage of $\Sel^{2^{\infty}} E^d$ in $H^1(G_{\QQ}, E^d[\varphi])$. Taking $\varphi': E_0 \to E$ to be the dual isogeny to $\varphi$ and defining its Selmer group analogously, we have
\begin{equation}
\label{eq:Tama}
u(\varphi, d) = \dim_{\FFF_2} \Sel^{\varphi}(E^d) - \dim_{\FFF_2} \Sel^{\varphi'}(E_0^d),
\end{equation}
as follows from the Greenberg--Wiles' formula \cite[Theorem 8.7.9]{Neuk08}.

In the case that $\varphi$ is balanced, the long exact sequence associated to
\begin{equation}
\label{eq:varphi_ES_loc}
0 \to E^d[\varphi]  \to E^d[2^{\infty}] \to E_0^d[2^{\infty}] \to 0
\end{equation}
for the group $G_{\QQ}$ gives
\[\dim_{\FFF_2} \Sel^{\varphi} E^d = 1 + r_{\varphi}(E^d)\quad\text{if}\quad E^d(\QQ)[2] = E^d(\QQ)[4] \quad\text{and}\quad E_0^d(\QQ)[2] = E_0^d(\QQ)[4].\]
If $p$ is an odd prime where $d$ has odd valuation but $E$ is good, the long exact sequence associated to \eqref{eq:varphi_ES_loc} over $G_p$ yields the exact sequence
\begin{equation}
\label{eq:balanced}
0 \to \FFF_2 \to H^0(G_{p}, E^d[2]) \to H^0(G_{p}, E_0^d[2]) \to W_p(\varphi, d) \to 0.
\end{equation}
Since $\varphi$ is balanced, the two middle terms in this sequence are isomorphic, and we see that $\#W_p(\varphi, d) = 2$.

So we find that the $u(\varphi, d_0)$ given in Notation \ref{notat:early_Tama} is well defined and equals the $u(\varphi, d_0)$ constructed above.
\end{defn}

\begin{rmk}
If $\varphi$ were not balanced, the middle terms in \eqref{eq:balanced} would not need to be isomorphic, and indeed often would not be isomorphic. This is what leads to the degenerate behavior of the Selmer groups associated to a non-balanced isogeny in a quadratic twist family.
\end{rmk}

\begin{center}
\begin{figure}

\begin{tabular}{c c c c c } 
\begin{tikzpicture}
\filldraw[white](0pt, 0pt) circle (2pt);
\filldraw[white](-10pt, 0pt) circle (2pt);

\filldraw[white](10pt, 45pt) circle (2pt);

\filldraw[black](0pt,20pt) circle (3.2pt);

\end{tikzpicture} &

\begin{tikzpicture}
\filldraw[white](-10pt, 0pt) circle (2pt);
\filldraw[white](35pt, 45pt) circle (2pt);

\draw[very thick] (0pt, 20pt) -- (25pt, 20pt);
\filldraw[black](0pt,20pt) circle (3.2pt);
\filldraw[black](25pt,20pt) circle (3.2pt);

\end{tikzpicture} &

\begin{tikzpicture}
\filldraw[white](-10pt, 0pt) circle (2pt);
\filldraw[white](30pt, 45pt) circle (2pt);

\filldraw[black](0pt, 20pt) circle (3.2pt);
\filldraw[black](20pt, 20pt) circle (3.2pt);
\filldraw[black](-10pt, 37.3pt) circle (3.2pt);
\filldraw[black](-10pt, 2.7pt) circle (3.2pt);

\draw[very thick] (3pt, 20pt) -- (17pt, 20pt);
\draw[very thick] (-1.5pt, 17.4pt) -- (-8.5pt, 5.3pt);
\draw[very thick] (-1.5pt, 22.6pt) -- (-8.5pt, 34.7pt);

\end{tikzpicture} &
\begin{tikzpicture}
\filldraw[white](-20pt, 0pt) circle (2pt);
\filldraw[white](50pt, 45pt) circle (2pt);

\filldraw[black](0pt, 20pt) circle (3.2pt);
\filldraw[black](30pt, 20pt) circle (3.2pt);
\filldraw[black](-10pt, 37.3pt) circle (3.2pt);
\filldraw[black](-10pt, 2.7pt) circle (3.2pt);
\filldraw[black](40pt, 37.3pt) circle (3.2pt);
\filldraw[black](40pt, 2.7pt) circle (3.2pt);

\draw[very thick] (3pt, 20pt) -- (27pt, 20pt);
\draw[very thick] (-1.5pt, 17.4pt) -- (-8.5pt, 5.3pt);
\draw[very thick] (-1.5pt, 22.6pt) -- (-8.5pt, 34.7pt);
\draw[very thick] (31.5pt, 17.4pt) -- (38.5pt, 5.3pt);
\draw[very thick] (31.5pt, 22.6pt) -- (38.5pt, 34.7pt);

\end{tikzpicture}
&
\begin{tikzpicture}
\filldraw[white](0pt, 0pt) circle (2pt);
\filldraw[white](80pt, 45pt) circle (2pt);

\filldraw[black](20pt, 20pt) circle (3.2pt);
\filldraw[black](60pt, 20pt) circle (3.2pt);
\filldraw[black](40pt, 20pt) circle (3.2pt);

\filldraw[black](10pt, 37.3pt) circle (3.2pt);
\filldraw[black](10pt, 2.7pt) circle (3.2pt);
\filldraw[black](70pt, 37.3pt) circle (3.2pt);
\filldraw[black](70pt, 2.7pt) circle (3.2pt);
\filldraw[black](70pt, 2.7pt) circle (3.2pt);
\filldraw[black](40pt, 40pt) circle (3.2pt);

\draw[very thick] (23pt, 20pt) -- (57pt, 20pt);
\draw[very thick] (18.5pt, 17.4pt) -- (11.5pt, 5.3pt);
\draw[very thick] (18.5pt, 22.6pt) -- (11.5pt, 34.7pt);
\draw[very thick] (61.5pt, 17.4pt) -- (68.5pt, 5.3pt);
\draw[very thick] (61.5pt, 22.6pt) -- (68.5pt, 34.7pt);
\draw[very thick] (40pt, 23pt) -- (40pt, 37pt);
\end{tikzpicture}

\end{tabular}
\caption{The $5$ possible forms of the $2$-isogeny graph of a rational elliptic curve.}
\label{fig:isogenies}
\end{figure}
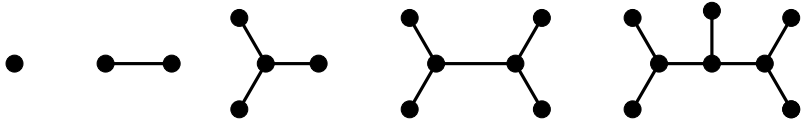
\end{center}

\subsection{The Cassels--Tate pairing on isogeny Selmer groups}
Given a $\QQ$-isogeny $\varphi: E \to E_0$ of degree $2$ with dual isogeny $\varphi'$ and a nonzero integer $d$, we may consider the Cassels--Tate pairing associated as in \cite{MS21} to the exact sequence
\[0 \to E^d[\varphi] \to E^d[2] \xrightarrow{\,\,\varphi\,\,}E_0^d[\varphi'] \to 0\]
of Galois modules. This gives an alternating pairing
\[\CTP\,:\, \Sel^{\varphi'} E_0^d \times \Sel^{\varphi'} E_0^d  \to \tfrac{1}{2}\Z/\Z\]
whose kernel equals $\varphi(\Sel^2 E^d)$ \cite[Section 6.1]{MS21}. This kernel necessarily contains the kernel of the map from $\Sel^{\varphi'} E_0^d$ to  $\Sel^{2^{\infty}} E_0^d$, which equals the image of $H^0(G_{\QQ}, E^d[2])$ under the associated connecting map for all but finitely many squarefree $d$.

Our general expectation is that, as $d$ varies, this pairing should behave like a random alternating matrix over $\FFF_2$ subject to the restriction that the kernel contains the image of $H^0(G_{\QQ}, E^d[2])$. Theorem \ref{thm:2IV} shows that this expectation is correct if $E$ is in Case IV and $\varphi$ is the unique balanced isogeny from $E$. It is also generally true in cases where the isogeny Selmer groups are degenerate. For example, if $\varphi: E \to E_0$ is the $\QQ$ isogeny of degree $2$ from a curve in Case II, then among the $50\%$ of twists where $r_{\varphi}(E^d)$ is $0$, the $2$-Selmer rank of $E^d$ is $r$ with probability 
\[\lim_{n \to \infty} \tfrac{1}{2}P^{\text{Alt}}(r\,|\, r + 2n),\]
which is consistent with the Cassels--Tate pairing on the large groups $\Sel^{\varphi'} E_0^d$ being uniformly distributed.

This general expectation is not the case for elliptic curves in Case V. In this case, the form of the Cassels--Tate pairing on the isogeny Selmer groups is heavily restricted.

To help study curves in Case V, we reproduce in Figure \ref{fig:isogenies} the  five forms that the $2$-isogeny graph of an elliptic curve over $\QQ$ can take, as recorded in \cite{Chil21}. In these graphs, vertices represent elliptic curves over $\QQ$, and edges represent rational $2$-isogenies.

In these graphs, $E(\QQ)[2] = (\Z/2\Z)^2$ if and only if the vertex corresponding to $E$ has degree $3$. So we see that $E$ is in Case V only if it corresponds to the central node of the final graph, and that the other nodes of degree $3$ in this graph are in Case IV.

\begin{defn}
\label{defn:V_setup}
Suppose $E$ is a curve in Case V. From Figure \ref{fig:isogenies}, it has two distinct balanced isogenies
\[\varphi_1: E \to E_1\quad\text{and}\quad \varphi_2: E \to E_2,\]
where $E_1$ and $E_2$ are curves in Case IV. Take $\varphi'_1$ and $\varphi'_2$ to be the corresponding dual isogenies.

Given any rational place $v$ and nonzero integer $d$, we have
\[\varphi_1(W_v(\varphi_2, d)) \subseteq W_v(\varphi_1', d) \quad\text{and}\quad\varphi_2(W_v(\varphi_1, d)) \subseteq W_v(\varphi_2', d).\]
Take $\Vplac$ to be the set of bad places of $E$ together with $2$ and $\infty$.
Since $\varphi_1$ and $\varphi_2$ are balanced, we have
\[\#W_v(\varphi_1, d) = \#W_v(\varphi_2, d) = \#W_v(\varphi'_1, d) = \#W_v(\varphi'_2, d) = 2\]
for $v$ outside $\Vplac$, so the above inclusions are equalities for $v$ outside $\Vplac$.

As a result, if we define
\[\mathcal{L}(E_1^{d}) = \prod_{v \in \Vplac} W_v(\varphi_1', d)/ \varphi_1(W_v(\varphi_2, d)),\]
the natural localization maps define an injection
\[\text{loc}: \Sel^{\varphi_1'} E_1^d/ \varphi_1(\Sel^{\varphi_2} E^d) \hookrightarrow \mathcal{L}(E_1^d).\]
Note that, if $d$ lies in $X_E(d_0, \infty)$ for some nonzero integer $d_0$, we have
\[\mathcal{L}(E_1^d) = \mathcal{L}(E_1^{d_0}).\]
Given this, we may ask how the image of the localization map varies in $\mathcal{L}(E_1^{d_0})$ as $d$ varies through $X_E(d_0, \infty)$.
\end{defn}

\begin{thm}
\label{thm:2V_II}
Choose a curve $E/\QQ$ in Case V and a nonzero integer $d_0$. Define $\varphi_1: E \to E_1$, $\varphi_1'$ and $\mathcal{L} = \mathcal{L}(E_1^{d_0})$ as in Definition \ref{defn:V_setup}. Choose a subgroup $V$ of $\mathcal{L}$ and an integer $r \ge \max(0, -u(\varphi_1, d_0))$. Then
\[\lim_{H \to \infty} \frac{\# \left\{d \in X_E(d_0, H)\,:\,\, r_{\varphi_1'}(E_1^d) = r\, \text{ and }\,\, \textup{loc}(\Sel^{\varphi_1'} E_1^d) = V \right\}}{\# \left\{d \in X_E(d_0, H)\,:\,\, r_{\varphi_1'}(E^d) = r\right\}}\]
equals the probability a uniformly selected random homomorphism from $\FFF_2^r$ to $\mathcal{L}$ has image equal to $V$.
\end{thm}
We will prove this theorem in Section \ref{sec:ranks}.

The importance of this theorem comes from the following proposition.

\begin{prop}
\label{prop:CTP_V}
Choose a curve $E/\QQ$ in Case V and nonzero integer $d_0$. Define  $\varphi_1: E \to E_1$,  $\varphi'_1$, and $\mathcal{L} = \mathcal{L}(E_1^{d_0})$ as in Definition \ref{defn:V_setup}. Then there is a nondegenerate alternating pairing
\[\langle\,\,,\,\, \rangle\colon \mathcal{L} \times \mathcal{L} \to \tfrac{1}{2}\Z/\Z\]
such that, for any $d \in X_E(d_0, H)$ and $\phi, \psi \in \Sel^{\varphi'_1} E_1^d$,  we have
\[\CTP(\phi, \psi) = \langle \textup{loc}(\phi),\, \textup{loc}(\psi) \rangle,\]
where the Cassels--Tate pairing is associated to the sequence
\begin{equation}
\label{eq:VIII_CT}
0 \to E^d[\varphi_1] \to E^d[2] \xrightarrow{\,\,\varphi_1\,\,} E_1^d[\varphi_1'] \to 0
\end{equation}
\end{prop}
\begin{proof}
Defining $\varphi_2$ as in Definition \ref{defn:V_setup}, the isogeny $\varphi_1$ defines an isomorphism
\[E^d[\varphi_2] \isoarrow E_1^d[\varphi_1']\]
and the inverse of this map defines a section $s: E_1^d[\varphi_1'] \to E^d[2]$ to the final map in \eqref{eq:VIII_CT}. 

Choose a rational place $v$ and $\phi, \psi$ in $W_v(\varphi_1', d)$. We then may choose
\[\Phi, \Psi \in \ker\left(H^1(G_v, E^d[2]) \to H^1(G_v, E^d)\right)\]
so that $\varphi_1(\Phi) = \phi$ and $\varphi_1(\Psi) = \psi$. We then define
\[\langle \phi ,\psi \rangle_v := \inv_v\left( s(\phi) \cup (\Psi - s(\psi))\right)  = \inv_v\left(s(\psi) \cup (\Phi - s(\phi))\right) \in \tfrac{1}{2}\Z/\Z,\]
where the cup product is with respect to the Weil pairing on $E^d[2]$. That the identity holds follows since the Weil pairing is symmetric and since $\Phi - s(\phi)$ and $\Psi - s(\psi)$ both lie in $H^1(G_v, E[\varphi_1])$, implying
\[(\Phi - s(\phi)) \cup (\Psi - s(\psi)) = 0.\]
From this identity, we see that the pairing $\langle\,\,,\,\,\rangle_v$ is symmetric.

From \cite[Definition 6.2]{MS21}, we see that
\[\CTP(\phi, \psi) =\sum_v \langle \res_{G_v}(\phi),\,\res_{G_v}(\psi) \rangle_v \quad\text{for all }\, \phi, \psi \in \Sel^{\varphi_1'} E_1^d.\]
Furthermore, from the remark after \cite[(6.3)]{MS21}, the left kernel of the pairing $\langle\,\,,\,\,\rangle_v$ is $\varphi_1(W_v(\varphi_2, d))$. This is then its right kernel, and the sum of local pairings defines a nondegenerate symmetric pairing on $\mathcal{L}$. This pairing does not depend on the choice of $d$ in $X_E(d_0, H)$.

From Theorem \ref{thm:2V_II}, we know that $\textup{loc}$ is surjective onto $\mathcal{L}$ for some $d$ in $X_E(d_0, \infty)$. Since the Cassels--Tate pairing is alternating, the pairing on $\mathcal{L}$ must also be alternating.
\end{proof}

Proposition \ref{prop:CTP_V} and Theorem \ref{thm:2V_II} together imply Theorem \ref{thm:2V}, as we show next.
\begin{proof}[Proof of Theorem \ref{thm:2V}]
The space $\mathcal{L} = \mathcal{L}(E_1^d)$ has dimension $u_0$ for $d$ in $X_E(d_0, H)$; it follows that $u_0$ is a nonnegative even integer from Proposition \ref{prop:CTP_V}. Take $V_d$ to be the image of $\Sel^{\varphi'_1} E_1^d$ in $\mathcal{L}$, and take $W_d$ to be the subgroup of $V_d$ orthogonal to $V_d$ under the pairing of Proposition \ref{prop:CTP_V}. We then have an exact sequence
\[0 \to \Sel^{\varphi_1} E^d \to \Sel^2 E^d \to \Sel^{\varphi'_1} E_1^d \to V_d/W_d \to 0,\]
so
\[r_2(E^d) = r_{\varphi_1}(E^d) + r_{\varphi'_1}(E_1^d) - \dim V_d/W_d\]
for all but finitely many $d$ in $X_E(d_0, H)$. So
\[r_2(E^d)  - r_{\varphi_1}(E^d) = r_{\varphi_1}(E^d)  - u_1 - \dim V_d/W_d\]
for all but finitely many such $d$.

If we choose a surjective homomorphism from $\FFF_2^{r_{\varphi_1}(E^d) - u_1}$ to $V_d$, then the right-hand side of this identity equals the dimension of the kernel of the pairing on $\FFF_2^{r_{\varphi_1}(E^d) - u_1}$ pulled back from the nondenerate pairing on $\mathcal{L}$. Call the dimension of this kernel $j_d$.

By Theorem \ref{thm:2V_II}, if we vary among $d$ in $X_E(d_0, H)$ such that $r_{\varphi_1}(E^d) = r_{\varphi_1}$ and let $H$ tend to infinity, the $V_d$ will be distributed like the image of a uniformly selected random homomorphism $\FFF_2^{r_{\varphi_1} - u_1} \to \mathcal{L}$. So, in the same family, we have $j_d = j$ with probability tending to $P^{\text{V}}(j\,|\, (r_{\varphi_1} - u_1) \to u_0)$ for any nonnegative integer $j$. This establishes Theorem \ref{thm:2V}.
\end{proof}

\subsection{Classifying cofavored submodules}
Given an elliptic curve $E/\QQ$, the first step in the method of \cite{Smi22b} to finding the distribution of $2$-Selmer groups in the twist family of $E$ is to classify the cofavored submodules of the powers $E^{\oplus a}$ for all nonnegative integers $a$, which are defined to be the submodules $T$ of $E^{\oplus a}[2]$ so that
\[\#H^0(\langle \sigma \rangle , (E^{\oplus a}/T)[2]) = \# H^0(\langle \sigma \rangle , E^{\oplus a}[2]) \quad\text{for all } \sigma \in G_{\QQ}.\]
If $E$ is in Case IV or V, the power $E^{\oplus a}$ has extra cofavored submodules not seen for curves in Case I for any positive integer $a$, and this corresponds to the nondegenerate isogeny component in the Selmer groups of the twist family. To actually calculate the moments of the sizes of the $2$-Selmer groups, we need to rule out extra cofavored submodules on top of the ones we expect.

To start, we will try to understand cofavored submodules of powers of elliptic curves more generally.

\begin{notat}
\label{notat:RE}
Choose an elliptic curve $E/\QQ$, and fix an identification $E[2^{\infty}] \cong (\QQ_2/\Z_2)^2$. From this identification, we may define a representation
\[\rho_E: G_{\QQ} \to \text{GL}_2(\Z_2)\]
corresponding to the Galois action on $E[2^{\infty}]$. From the Weil pairing, we find
\begin{equation}
\label{eq:det_rho}
\det(\rho_E(\sigma)) \equiv \frac{\sigma(\sqrt{-1})}{\sqrt{-1}}\, \textup{ mod } 4 \quad\text{ for all }\sigma \in G_{\QQ}.
\end{equation}

Take $M_2(\FFF_2)$ to be the ring of $2 \times 2$ matrices over $\FFF_2$, and take $\phi_E: G_{\QQ(E[2])} \to M_2(\FFF_2)$ to be the map defined by
\[\phi_E(\sigma) = \tfrac{1}{2}(\rho_E(\sigma) - 1).\]
We take $R_E$ to be the subring with identity generated by the image of $\phi_E$.

For convenience, we name
\[\alpha = \begin{pmatrix} 0 & 1 \\ 0 & 0 \end{pmatrix}, \quad \beta = \begin{pmatrix} 1 & 0 \\ 0 & 0 \end{pmatrix},\quad\text{and}\quad\text{Id} = \begin{pmatrix} 1 &0 \\ 0 & 1 \end{pmatrix}.\]
\end{notat}

\begin{prop}
Choose an elliptic curve $E/\QQ$ and a positive integer $a$, and fix an identification of $E[2^{\infty}]$ with $(\QQ_2/\Z_2)^2$. Via this identification, we may view $E[2]^{\oplus a}$ as an $M_2(\FFF_2)$ module. Given a $G_{\QQ}$-submodule $T$ of $E[2]^{\oplus a}$, if $T$ is cofavored, then $T$ is closed under the action of $R_E$.
\end{prop}
\begin{proof}
Choose any $\sigma \in G_{\QQ(E[2])}$. Since $T$ is cofavored, we have
\[\# H^0(\langle \sigma \rangle, (E^{\oplus a}/T)[2]) = \# H^0(\langle \sigma \rangle, E^{\oplus a}[2]) = 4^a,\]
so $\sigma$ acts trivially on $(E^{\oplus a}/T)[2]$. This implies that $\phi_E(\sigma)$ maps $T$ into $T$. Applying this for all $\sigma \in G_{\QQ(E[2])}$ then gives the result.
\end{proof}

\begin{prop}
\label{prop:IV_cof}
Suppose $E$ is an elliptic curve in Case IV, take $\varphi: E \to E_0$ to be its balanced isogeny, and take $\varphi'$ to be the dual isogeny to $\varphi$. Choose nonnegative integers $a$ and $b$, and take $T$ to be a submodule of
\[M = E[\varphi]^{\oplus a} \oplus E[2]^{\oplus b}\]
such that the projection $M \to E[\varphi]^{\oplus a}$ restricts to a surjection on $T$, and such that the composition 
\[M \to 0 \oplus E[2]^{\oplus b} \xrightarrow{\,\,\varphi^{\oplus b}\,\,} E_0[\varphi']^{\oplus b}\]
is surjective when restricted to $T$. Then, if $T$ is cofavored, it equals $M$.
\end{prop}
\begin{proof}
Fix an identification $\iota$ of $E[2^{\infty}]$ with $(\QQ_2/\Z_2)^2$ such that the nontrivial point in $E[\varphi]$ maps to $(1/2, 0)$. Since $\varphi$ is balanced, we see that $R_E$ has image in
\[R_{\text{IV}} := \langle \alpha, \beta, \text{Id}\rangle.\]
We claim that this equals $R_E$. To start, we note that $R_E$ contains $\text{Id}$.

First consider the case that $\QQ(E[2]) = \QQ$. Since $\varphi$ is the unique balanced isogeny of $E$, we find that $R_E$ cannot fix the submodules $\langle (0, 1/2) \rangle$ or $\langle (1/2, 1/2)\rangle$. So it contains
\[a_1 \beta + \alpha \quad\text{and}\quad a_2(\alpha + \beta) + \alpha\]
for some $a_1, a_2 \in \FFF_2$. It follows that it contains $\alpha$.

 Choose some $\tau \in G_{\QQ}$ that does not fix $\sqrt{-1}$. As $E$ is in Case IV, $\tau\circ \iota (1/4, 0)$ is a multiple of $\iota(1/4, 0)$.
Together with \eqref{eq:det_rho}, we find that $\phi_E(\tau)$ has the form
\begin{equation}
\label{eq:IV_Weil}
\begin{pmatrix} a & b \\ 0 & 1 + a\end{pmatrix},
\end{equation}
for some $a, b \in \FFF_2$, and the claim that $R_E = R_{\text{IV}}$ follows.

Now suppose that $\QQ(E[2]) \ne \QQ$. In this case, there are integers $a, b$ such that
\[\QQ(E[2]) = \QQ(\sqrt{a^2 - 4b}) \quad\text{and}\quad \QQ(E_0[2]) = \QQ(\sqrt{b});\]
see \cite[Example 1.2]{Smi22b}.  So this field cannot equal $\QQ(\sqrt{-1})$, and we find that $R_E$ contains an element of the form \eqref{eq:IV_Weil}.

Now choose $\sigma$ acting nontrivially on $E[2]$ and not fixing $\sqrt{-1}$, so $\sigma(\iota(0, 1/2))$ is $\iota(1/2, 1/2)$. It also acts nontrivially on $E_0[2]$, so $\sigma(\iota(1/4, 0))$ is either $\iota(1/4, 1/2)$ or $\iota(3/4, 1/2)$.  So
\[\rho_E(\sigma) = \begin{pmatrix} 1+2a & 1 + 2b \\ 2 + 4c & 1+2d \end{pmatrix}\]
for some $2$-adic integers $a, b, c, d$. From \eqref{eq:det_rho}, we see that $a$ equals $d$ mod $2$.

Considered mod $4$, $\rho_E(\sigma^2)$ has the form
\[\begin{pmatrix} 3  & 2 \\ 0 & 3 \end{pmatrix},\]
so $R_E$ contains
\[\begin{pmatrix} 1 & 1 \\ 0 & 1\end{pmatrix}.\]
This is enough to prove $R_E = R_{\text{IV}}$.

Now consider $T$ as in the proposition statement. It is a $R_{\text{IV}}$ submodule of $M$. We note that
\[\alpha T = 0 \oplus E[\varphi]^{\oplus b}\]
since $T$ surjects onto $E_0[\varphi']^{\oplus b}$. Similarly, we see that 
\[\beta T\subseteq T\cap E[\varphi]^{\oplus a + b}\]
 surjects onto $E[\varphi]^{\oplus a}$. So $T$ must contain $E[\varphi]^{\oplus a + b}$. Again using the fact that it surjects onto $E_0[\varphi']^{\oplus b}$, we may conclude $T = M$.
\end{proof}

\begin{prop}
\label{prop:V_cof}
Suppose $E$ is an elliptic curve in Case V, and take $\varphi_i: E \to E_i$ for $i = 1,2$ to be its balanced isogenies. Take $\varphi'_i$ to be the dual isogeny to $\varphi_i$.

Choose nonnegative integers $a$, $b$, and $c$, and take $T$ to be a submodule of
\[M  = E[\varphi_1]^{\oplus a} \oplus E[\varphi_2]^{\oplus b} \oplus E[2]^{\oplus c}.\]
We assume that the compositions
\[M \to E[\varphi_1]^{\oplus a} \oplus E[2]^{\oplus c} \xrightarrow{\,\,\varphi_2\,\,} E_2[\varphi'_2]^{\oplus a + c}\quad\text{and}\quad M \to E[\varphi_2]^{\oplus b} \oplus E[2]^{\oplus c}  \xrightarrow{\,\,\varphi_1\,\,} E_1[\varphi'_1]^{\oplus b +c}\]
remain surjective when restricted to $T$. Then, if $T$ is cofavored, it equals $M$.
\end{prop}
\begin{proof}
Fix an identification $\iota$ of $E[2^{\infty}]$ with $(\QQ_2/\Z_2)^2$ such that the nontrivial point in $E[\varphi_1]$ maps to $(1/2, 0)$ and the nontrivial point in $E[\varphi_2]$ maps to $(0, 1/2)$. With this identification, we see that $R_E$ is contained in 
\[R_{\text{V}} := \langle\text{Id}, \beta\rangle.\]
From \eqref{eq:det_rho}, we see that $\rho_E(\tau)$ is either $\beta$ or $\text{Id} + \beta$ for any $\tau \in G_{\QQ}$ not fixing $\sqrt{-1}$. So $R_E = R_{\text{V}}$.

Now consider $T$ as in the proposition statement. It must be an $R_{\text{V}}$ module. Since it surjects onto $E_2[\varphi'_2]^{\oplus a+c}$, we see that
\[\beta T = E[\varphi_1]^{\oplus a} \oplus 0 \oplus E[\varphi_1]^{\oplus c}.\]
Similarly,
\[(1 + \beta) T = 0 \oplus E[\varphi_2]^{\oplus b} \oplus E[\varphi_2]^{\oplus c},\]
and the result follows.
\end{proof}

Our final result will be needed to classify non-cancellable pairings in the calculation of moments, as defined in \cite[Section 7]{Smi22b}.
\begin{prop}
\label{prop:non-canc}
Suppose $E$ is an elliptic curve in Case IV, and take $\varphi: E \to E_0$ to be its balanced isogeny. View $E[2]$ as an $R_E$ module as in Notation \ref{notat:RE}.

Then the only nonzero $R_E$-equivariant homomorphism
\[\Gamma:E[2] \to E[2]\]
is the identity; the only nonzero $R_E$-equivariant homomorphism
\[\Gamma: E[2] \to E[2]/E[\varphi]\]
is the standard projection; the only nonzero $R_E$-equivariant homomorphism
\[\Gamma: E[\varphi] \to E[2]\]
is the standard inclusion; and there is no nonzero $R_E$-equivariant homomorphism
\[\Gamma: E[\varphi] \to E[2]/E[\varphi].\]
\end{prop}
\begin{proof}
As proved in Proposition \ref{prop:IV_cof}, $R_E$ contains $\alpha$ and $\beta$.

For the first claim, we calculate
\begin{align*}
&\beta\cdot  \begin{pmatrix} a & b \\ c& d \end{pmatrix} = \begin{pmatrix} a & b \\ 0 & 0 \end{pmatrix},\quad  \begin{pmatrix} a & b \\ c& d \end{pmatrix}\cdot \beta = \begin{pmatrix} a & 0 \\ c & 0 \end{pmatrix},\\
&  \alpha \cdot \begin{pmatrix} a& b \\ c& d \end{pmatrix} = \begin{pmatrix} c & d \\ 0 & 0 \end{pmatrix},\quad \begin{pmatrix} a & b \\ c & d \end{pmatrix}\cdot \alpha = \begin{pmatrix} 0 & a \\ 0 & c \end{pmatrix}
\end{align*}
for $a, b, c, d \in \FFF_2$. For $\begin{pmatrix} a & b \\ c& d \end{pmatrix}$ to commute with $\alpha$ and $\beta$, we thus need $b = c = 0$ and $a =d$. So the only $R_E$ equivariant maps $\Gamma: E[2] \to E[2]$ are multiples of the identity.

For the final claim, we note that $\beta$ acts as the identity on $E[\varphi]$ but is zero on $E[2]/E[\varphi]$, so the only homomorphism $\Gamma: E[\varphi] \to E[2]/E[\varphi]$ that commutes with $\beta$ is the zero map. The second and third claims then follow directly from the fourth claim. For the second claim, we see that the restriction of $\Gamma$ to $E[\varphi]$ must be trivial, so $\Gamma$ factors through the projection to $E[2]/E[\varphi]$; similarly, for the third claim, we see the composition of $\Gamma$ with the projection to $E[2]/E[\varphi]$ is trivial, so $\Gamma$ has image in $E[\varphi]$.
\end{proof}

\section{Moments of $2$-Selmer groups in grids of quadratic twists}
\label{sec:grids}
The goal of this section and the next is to determine the distribution of $2$-Selmer groups in the quadratic twist family of an elliptic curve in Case IV or V. From the work in the previous section, we know that it suffices to prove Theorems \ref{thm:balanced}, \ref{thm:2IV}, and \ref{thm:2V_II}.

These calculations start with the calculation of a corresponding set of refined $2$-Selmer moments. We will approach this work as in \cite{Smi22b}.

\begin{defn}
\label{defn:Sab}
Take $E$ to be an elliptic curve in Case IV, and take $\varphi: E \to E_0$ to be its balanced isogeny. Take $\Vplac_0$ to be a finite set of rational places containing the places of bad reduction for $E$ together with $2$ and $\infty$.

Choose a nonzero integer $d_0$. Given $d$ in $X_E(d_0, \infty)$, we will be interested in two sets of tuples of Selmer elements for the twist $E^d$.

First, given nonnegative integers $a, b$, we take $\mathcal{S}_{a, b}(E^{d})$ to be the set of tuples
\[(\phi_1, \dots, \phi_{a+b}) \in \left(\Sel^{\varphi} E^d\right)^{\oplus a} \oplus \left(\Sel^2 E^d\right)^{\oplus b}\]
such that $\phi_1, \dots, \phi_a$ have linearly independent images in $\Sel^{2^{\infty}} E^d$, and such that $\varphi(\phi_{a+1})$, $\dots$, $\varphi(\phi_{a+b})$ have linearly independent images in $\Sel^{2^{\infty}} E_0^d$.

Second, choose a nonnegative integer $a$, and choose a subspace
\[V \subseteq \prod_{v \in \Vplac_0}W_v(\varphi, d_0)^{\oplus a}.\]
Then we take $\mathcal{S}_{a, V}(E^d)$ to be the set of tuples $(\phi_1, \dots, \phi_a)$ in $\mathcal{S}_{a, 0}(E^d)$ such that $(\phi_1, \dots, \phi_a)$ localizes to an element in the subspace $V$.
\end{defn}

\begin{thm}
\label{thm:moments}
Fix $\varphi: E \to E_0$, $\Vplac_0$, and $d_0$ as in Definition \ref{defn:Sab}. Take $u = u(\varphi, d_0)$.

Choose nonnegative integers $a, b$. Then
\[\lim_{H \to \infty} \frac{1}{\# X_E(d_0, H)} \sum_{d \in X_E(d_0, H)} \# \mathcal{S}_{a, b}(E^d) = \#\left(H^0(G_{\QQ}, E[2])\right)^b \cdot 2^{a + a u + ab + \frac{1}{2}b(b+1)}.\]
Now choose a nonnegative integer $a$ and a subspace $V$ of $\prod_{v \in \Vplac_0}W_v(\varphi, d_0)^{\oplus a}$. Take $k$ to be the codimension of this subspace. Then
\[\lim_{H \to \infty} \frac{1}{\# X_E(d_0, H)} \sum_{d \in X_E(d_0, H)} \# \mathcal{S}_{a, V}(E^d) = 2^{a + au - k}.\]
\end{thm}

Following \cite{Smi22b}, we prove this theorem by proving a variant of it on \emph{grids of twists}, as defined in \cite[Section 8]{Smi22b}.
\begin{notat}
\label{notat:grid}
Fix $\varphi: E \to E_0$ and $\Vplac_0$ as in Definition \ref{defn:Sab}. Take $L$ to be the maximal abelian extension of $\QQ(E[2])$ of exponent $2$ that is unramified at all primes outside $\Vplac_0$. We will assume that $\Vplac_0$ is large enough that $(\QQ(E[2])/\QQ, \Vplac_0, \pm 1)$ is an unpacked starting tuple, in the notation of \cite[Section 3]{Smi22a}.

Fix a squarefree integer $d_0$ indivisible by any prime outside $\Vplac_0$. Choose a finite set $S$, and for each $s \in S$, take $X_s$ to be a finite collection of primes. We assume the sets $X_s$ are all disjoint, and we take $X = \prod_{s \in S} X_s$. Given $x \in X$, we then define an integer
\[\chi(x) = d_0 \cdot \prod_{s \in S} \pi_s(x),\]
where $\pi_{s}(x)$ is the $s^{th}$-component of $x$.

Choosing some $H > 0$, we will assume that $X$ obeys the conditions of \cite[Definition 8.3 and Notation 8.4]{Smi22b} to be a good grid of filtered ideals of height $H$ with respect to $(\QQ(E[2]), \Vplac_0, \pm1)$. The definition of a good grid of filtered ideals has many parts; among the most salient properties are the following.
\begin{enumerate}
\item For all $x \in X$, $\chi(x)/d_0$ is at most $H$.
\item For every $s \in S$, the conjugacy class of the Frobenius element $\Frob\, p$ in $\Gal(L/\QQ)$ does not depend on the choice of $p$ in $X_s$.
\item For every conjugacy class in $\Gal(L/\QQ)$, there is an $s \in S$ so the Frobenius elements of the primes in $X_s$ land in the class.
\end{enumerate}
For any $x, x_0 \in X$, we find that $\chi(x)$ lies in $X_E(\chi(x_0), \infty)$. In particular, $u(\varphi, \chi(x))$ does not depend on the choice of $x$.
\end{notat}

\begin{prop}
\label{prop:moments}
Fix $\varphi: E \to E_0$ and $\Vplac_0$ as in Definition \ref{defn:Sab}. From this data, fix a good grid $X$ of height $H > 20$ as in Notation \ref{notat:grid}. Choose $x_0 \in X$; taking $\varphi$ to be the balanced isogeny on $E$, take $u = u(\varphi, d_0)$. Then, given nonnegative integers $a, b$ satisfying $a + b \le (\log \log H)^{1/9}$, we have
\[\frac{1}{\# X} \sum_{x \in X}\#\mathcal{S}_{a,b}(E^{\chi(x)}) = A + O \left( \exp\left(-(\log \log H)^{1/10}\right)\right),\]
where
\[A = \# H^0(G_{\QQ}, E[2])^b \cdot 2^{a + au + ab  + \tfrac{1}{2}b(b+1)}.\]
Furthermore, given a nonnegative integer $a \le (\log \log H)^{1/9}$ and a subspace $V$ of $\prod_{v \in \Vplac_0} W_v(\varphi, \chi(x_0))^{\oplus a}$, if we take $k$ to be the codimension of the subspace $V$, we have
\[\frac{1}{\# X} \sum_{x \in X}\#\mathcal{S}_{a,V}(E^{\chi(x)}) = 2^{a + au - k} + O \left( \exp\left(-(\log \log H)^{1/10}\right)\right).\]
Here, the implicit constants depend just on $E$ and $\Vplac_0$.
\end{prop}
The machinery to prove this proposition is largely developed in Sections 7 and 9 of \cite{Smi22b}. To start, we need to reinterpret the condition of lying in $\mathcal{S}_{a, b}(E^{\chi(x)})$ using the terminology of that paper.
\begin{defn}[{\cite[Definition 5.2]{Smi22b}}]
Choose $x = (x_s)_s \in X$. For $s \in S$, take $\sigma_s$ to be a topological generator for the tame part of the inertia group of some prime of $\ovQQ$ over $x_s$. Given $\Phi$ in $\Sel^2 (E^{\chi(x)})^{\oplus m}$, we define the ramification subspace $T(\Phi)$ of $\Phi$ to be the $G_{\QQ}$-submodule of $E[2]^{\oplus m}$ generated by all elements of the form
\[\tau \Phi(\sigma_s) - \Phi(\sigma_t) \quad\text{with}\quad s, t \in S,\,\,\, \tau \in G_{\QQ}.\]
\end{defn}
\begin{lem}
\label{lem:ram_sub}
Given $a, b \ge 0$, given $x = (x_s)_s$ in $X$, and given $\Phi = (\phi_1, \dots, \phi_{a+b})$  in $\Sel^2 (E^{\chi(x)})^{\oplus a+b}$, $\Phi$ lies in $\mathcal{S}_{a, b}(E^{\chi(x)})$ if and only if the following holds:
\begin{itemize}
\item Taking $Q = E[\varphi]^{\oplus a} \oplus E[2]^{\oplus b}$, the element $\Phi$ is in the image of $H^1(G_{\QQ}, Q)$;
\item The projection from $Q$ to $E[\varphi]^{\oplus a}$ restricts to a surjection on $T(\phi)$; and
\item Taking $\varphi': E_0 \to E$ to be the dual isogeny to $\varphi$, the projection from $Q$ to $E_0[\varphi']^{\oplus b}$ also restricts to a surjection on $T(\Phi)$.
\end{itemize}

\end{lem}
\begin{proof}
If these conditions are satisfied, it is clear that $\Phi$ lies in $\mathcal{S}_{a, b}(E^{\chi(x)})$. Conversely, suppose that $\Phi$ lies in $\mathcal{S}_{a, b}(E^{\chi(x)})$. The first condition is satisfied, so we just need to check that $T(\Phi)$ surjects onto $E[\varphi]^{\oplus a}$ and $E_0[\varphi']^{\oplus b}$.

Suppose to start that $T(\Phi)$ does not surject onto $E[\varphi]^{\oplus a}$, so there is a nonzero linear map $\lambda: E[2]^{\oplus a} \to E[2]$ mapping $T(\Phi)$ to $0$. Applying this to $\Phi$ gives an element $\lambda(\Phi)$ in $\Sel^2 E^{\chi(x)}$ lifting to $\psi \in H^1(G_{\QQ}, E[\varphi])$ such that $T(\lambda(\Phi)) =0$. After adjusting $\psi$ by an element of the kernel to the map to $\Sel^{2^{\infty}} E^{\chi(x)}$ if necessary, we may assume that $\psi$ is unramified at $x_s$ for all $s \in S$. By the definition of $\mathcal{S}_{a, b}$, we know that $\psi$ is nonzero. Since $\psi$ is unramified outside $\Vplac_0$, it must be the inflation of a cocycle on $\Gal(L/\QQ)$.

For $s \in S$, we have 
\[\left(E^{\chi(x)}[2^{\infty}]\right)^{I_{x_s}} = E^{\chi(x)}[2]\]
From the inflation-restriction exact sequence, we then know that the image of $\psi$ in
\[H^1(G_{x_s}/I_{x_s}, E[2])\]
is zero for all $s \in S$. From the third property of good grids of ideals enumerated in Notation \ref{notat:grid}, we then find that a cocycle representing $\psi$ is $0$ at all $\sigma$ in $\Gal(L/\QQ(E[2]))$.  So $\psi$ has trivial restriction to $\QQ(E[2])$.

Since $\psi$ is nonzero, $\QQ(E[2])/\QQ$ must be a nontrivial quadratic extension and $\psi$ must be the associated quadratic character. But then $\lambda(\Phi)$ was $0$, and we have a contradiction. The proof that the final condition holds is similar.
\end{proof}

\begin{proof}[Proof of Proposition \ref{prop:moments}]
Take $Q = E[\varphi]^{\oplus a} \oplus E[2]^{\oplus b}$, and take $H^1_{\Vplac}(G_{\QQ}, Q)$ to be the set of elements in $H^1(G_{\QQ}, Q)$ unramified outside the places in $\Vplac$ for any set of places $\Vplac$ of $\QQ$. Taking $\Frob \,X_s$ to be the conjugacy class of $\Gal(L/\QQ)$ represented by elements in $X_s$, we may define a parameterization
\[\Psi_{x}: H^1_{\Vplac_0}(G_{\QQ}, Q) \times \prod_{s \in S} H^0(\langle \Frob\, X_s\rangle,\, Q) \isoarrow H^1_{\Vplac_0 \cup \{x_s\,:\, s\in S\}}(G_{\QQ}, Q)\]
as in \cite[Section 5]{Smi22b}. Take $\mathscr{M}_0$ to be the subgroup of the $(\phi_0, (m_s)_s)$ in the domain mapping to cocycle classes obeying the local conditions for $E^{\chi(x)}$ at all the primes in $\Vplac_0$ for any/every $x \in X$ such that the image of $m_s$ is $0$ in $E[\varphi]^{\oplus a}$ if $\Frob\, X_s$ does not fix $E[2]$. By \cite[(5.3)]{Smi22b}, every element in $\mathcal{S}_{a, b}(E^{\chi(x)})$ is parameterized by an element in $\mathscr{M}_0$.

The kernel of the restriction map
\[H^1_{\Vplac_0}(G_{\QQ}, Q^{\vee}) \to \text{Hom}(G_{\QQ(E[2])}, Q^{\vee})\]
is of order $2^a$ if $G_{\QQ}$ acts nontrivially on $E[2]$ and otherwise equals $1$.  From the third property of the good grid of ideals enumerated in Notation \ref{notat:grid}, given $x = (x_s)_s$ in $X$, the elements in $H^1_{\Vplac_0}(G_{\QQ}, Q^{\vee})$  that are trivial at $x_s$ for all $s \in S$ such that $\Frob\,X_s$ acts trivially on $E[2]$ are precisely the elements in this kernel. From this together with Wiles' formula \cite[Theorem 8.7.9]{Neuk08}, we find that
\[\#\mathscr{M}_0 = 2^{au} \cdot\frac{4^a}{(\# H^0(G_{\QQ}, E[2]))^a} \cdot \prod_{s \in S}2^{-a} \cdot \left(\#H^0(\langle \Frob\, X_s\rangle,\, E[2])\right)^{a+b}.\]

We take $\mathscr{M}_1$ to be the subset of $\mathscr{M}_0$ parameterizing cocycle classes such that the corresponding ramification subspaces surject onto $E[\varphi]^{\oplus a}$ and $E_0[\varphi']^{\oplus b}$, and we take $\mathscr{M}_2$ to be the subset of these elements whose corresponding ramification subspace is $Q$.

From the definition of a good grid \cite[Notation 8.4 (4) and (5)]{Smi22b}, we find there is some $c > 0$ determined from $E$ and $\Vplac_0$ such that the number of $s \in S$ such that $\Frob\, X_s$ acts trivially on $E[2]$ is at least $c \log \log H/ \log \log \log H$. Given $\epsilon > 0$, it follows that
\[\#\mathscr{M}_2/ \# \mathscr{M}_0  = 1 + O_{\epsilon, E}\left(\exp\left(- (\log \log H)^{1 - \epsilon}\right)\right),\]

Take 
\[\mathscr{R} = \prod_{s \in S} H^0(\langle \Frob\, X_s\rangle,\, E[2]^{\oplus a + b}).\]
As in \cite[Section 5.3]{Smi22b}, given $x \in X$, we may define a pairing
\[\langle\,\,,\,\,\rangle_x\colon \mathscr{R} \times \mathscr{M}_0 \to \tfrac{1}{2}\Z/\Z\]
such that $m \in \mathscr{M}_0$ parameterizes a Selmer element for the twist $\chi(x)$ if and only if it lies in the right kernel of this pairing. So, by Lemma \ref{lem:ram_sub}, our goal is to estimate
\begin{equation}
\label{eq:main_moment_char}
\frac{1}{\# X \cdot \# \mathscr{R}} \sum_{x \in X} \sum_{(m, r) \in \mathscr{M}_1 \times \mathscr{R}} \exp\left(2\pi i \langle m, r \rangle_x \right).
\end{equation}
As in the proof of \cite[Proposition 10.3]{Smi22b}, we may restrict this sum to $m$ such that the associated ramification subspace is cofavored with negligible effect on the sum. By Proposition \ref{prop:IV_cof}, these are precisely the $m$ in $\mathscr{M}_2$.

Following \cite[Proposition 10.3 and Proposition 7.7]{Smi22b}, we may restrict the inner sum to non-ignorable, non-borderline pairs $(m, r)$ with negligible effect on the sum; see \cite[Section 6]{Smi22b} for definitions of both of these terms. By \cite[Proposition 6.6]{Smi22b}, for any non-ignorable non-borderline $((\phi_0, (m_s)_s),  (r_s)_s)$, there is an alternating $G_{\QQ}$-equivariant map 
\[\Gamma: Q \to (E[2]/E[\varphi])^{\oplus a} \oplus E[2]^{\oplus b}\]
such that  $r_s - r_t$ projects to $\Gamma(m_s - m_t)$ for all $s, t \in S$. Given $m$ and $\Gamma$, we see that $r$ is determined up to an element of the form $(r_0 + r_{0s})_s$, where $r_{0s}$ lies in $E[\varphi]^{\oplus a}$ for each $s \in S$ and $r_0$ lies in $H^0(G_{\QQ}, E[2])^{\oplus (a+b)}$.

We note that
\[\langle m, (r_0)_s\rangle_x = 0\quad\text{and}\quad \langle m, (r_{0s})_s\rangle_x = 0.\]
for all $x \in X$ and $m \in \mathscr{M}_2$. The first claim here follows from \cite[Proposition 7.5]{Smi22b}. For the second, we note that the connecting map
\[H^0(G_{x_s}, E[\varphi]) \to H^1(G_{x_s}, E[2]/E[\varphi])\]
is trivial for a given $(x_s)_s \in X$ and $s \in S$ unless $\Frob\, X_s$ acts nontrivially on $E[2]$, in which case $n_s$ is trivial by the definition of $\mathscr{M}_0$, so the claim follows from \cite[(5.2)]{Smi22b}. So, for a non-ignorable non-borderline pair $(m, r)$ as above, we may determine the paring $\langle m, r \rangle_x$ from $m$ and $\Gamma$; as such, we may renotate this pairing in the form $\langle m, \Gamma(m)\rangle_x$ and rewrite \eqref{eq:main_moment_char} up to negligible error as
\[\frac{2^{a \cdot (\# S - 1)} \cdot (\#H^0(G_{\QQ}, E[2]))^{a+b}}{\# X \cdot \# \mathscr{R}} \sum_{x \in X} \sum_{m \in \mathscr{M}_2} \sum_{\Gamma} \exp\left(2\pi i \langle m, \Gamma(m) \rangle_x \right).\]
 By \cite[Proposition 7.3 and (7.3)]{Smi22b}, we may restrict the sum over $\Gamma$ to non-cancellable $\Gamma$ with negligible effect on the expression, where cancellable is defined as in \cite[Definition 7.2]{Smi22b}. Among the  $G_{\QQ}$-equivariant alternating maps
\[\Gamma: E[\varphi]^{\oplus a} \oplus E[2]^{\oplus b} \to (E[2]/E[\varphi])^{\oplus a} \oplus E[2]^{\oplus b},\]
the non-cancellable $\Gamma$ are recognizable as the $R_E$-equivariant homomorphisms, as considered in Proposition \ref{prop:non-canc}. We may consider $\Gamma$ as an alternating $(a+b) \times (a+b)$ matrix; for $\Gamma$ to be non-cancellable, the entries of this matrix must take the forms of Proposition \ref{prop:non-canc}. There are $2^{ab + \tfrac{1}{2}b(b+1)}$ such forms the matrix can take. If it takes one of these forms, we find that $\langle m, \Gamma(m)\rangle_x$ is $0$ by the logic of \cite[Proposition 7.4]{Smi22b}. The first part of the result now follows, and the second part is similar.
\end{proof}

\begin{proof}[Proof of Theorem \ref{thm:moments}]
Take $\mathscr{P}$ to be the set of squarefree integers divisible only by primes in $\Vplac_0$. Given $H_1 > 0$ and $d_1$ in $\mathscr{P}$, take $X^*(d_1, H_1)$ to be the set of squarefree numbers of the form $d_1 d$, where $d\le H$ is a positive integer indivisible by any prime in $\Vplac_0$ such that $d_1d$ lies in $X_E(d_0, \infty)$. Then $X_E(d_0, H)$ may be written as the disjoint union of the sets $X^*(d_1, H/|d_1|)$, so it suffices to find the moments over these sets under the condition that they are nonempty.

So suppose $X^*(d_1, H/|d_1|)$ is nonempty. It then may be written as the disjoint union of a collection of good grids of height $H/|d_1|$ together with the set of ``bad twists'' $Z$, where bad is defined as in \cite[Section 8]{Smi22b}. From \cite[Proposition 9.4]{Smi22b}, the bad twists are negligible in proportion to the size  of $X^*(d_1, H/|d_1|)$, and  the bad twists contribute negligibly to the total moment across $X^*(d_1, H/|d_1|)$ as $H$ tends to $\infty$ for a fixed choice of $a, b$ or $a, V$. So the moment may be estimated by taking the sum over the disjoint good grids in $X^*(d_1, H/|d_1|)$, and the result follows from Proposition \ref{prop:moments}.
\end{proof}

\section{From $2$-Selmer moments to $2$-Selmer ranks}
\label{sec:ranks}
Having estimated some refined moments across $X_E(d_0, H)$, our next goal is to reverse engineer the distribution of Selmer groups. To do this, we will prove that the distributions of $2$-Selmer groups claimed in Theorems \ref{thm:balanced}, \ref{thm:2IV}, and \ref{thm:2V_II} are consistent with our calculated moments. We then need to prove that these distributions of groups are the only ones that produce such moments. In \cite{Smi22b}, we proved the analogous uniqueness statement using a complex-analytic argument. Here, where effectivity of our estimates is not a priority, we take a slightly cleaner approach based on an abstract theorem of Sawin and Matchett Wood \cite{Saw22}.

\begin{defn}
Take $\mathscr{C}$ to be the category of tuples $(V_0, V)$, where $V$ is a finite vector space over $\FFF_2$ and $V_0$ is a subspace of $V$. A morphism $(V_0, V) \to (W_0, W)$ in this category will consist of an injection $V  \to W$ carrying $V_0$ into $W_0$ such that the corresponding homomorphism $V/V_0 \to W/W_0$ is also injective. We write the set of such morphisms as $\text{Inj}((V_0, V), (W_0, W))$.

Given $O = (V_0, V)$ in this category, we define $d_1(O) = \dim V_0$ and $d_2(O) = \dim V/V_0$. Given a second object $A$ in the category, we note that
\begin{equation}
\label{eq:InjC}
\#\Inj(O, A) = \#\Inj\left(\FFF_2^{d_1(O)}, \FFF_2^{d_1(A)}\right) \cdot \#\Inj\left(\FFF_2^{d_2(O)},\,\FFF_2^{d_2(A)}\right)\cdot 2^{d_2(O) d_1(A)},
\end{equation}
where, given $\FFF_2$ vector spaces $V$ and $W$, we have taken the notation $\Inj(V, W)$ for the set of linear injections from $V$ to $W$.

We write the set of isomorphism classes of $\mathscr{C}$ as $\mathscr{C}_{\cong}$, and we endow this set with the discrete topology.
\end{defn}

\begin{prop}
\label{prop:IV_consis}
Choose a nonzero integer $u$, and define a measure $\mu$ on $\mathscr{C}_{\cong}$ by
\[\mu((\FFF_2^a, \FFF_2^{a+b})) = P^{\textup{Mat}}(a\, |\, (\infty - u) \times \infty) \cdot P^{\textup{Alt}}(b\,|\, a - u)\]
for nonnegative integers $a, b$. Then
\[\int \#\Inj\left((\FFF_2^a, \FFF_2^{a+b}),  A\right) d\mu(A) = 2^{au + ab + \tfrac{1}{2}b(b+1)}.\]
\end{prop}
\begin{proof}
Take $O = (\FFF_2^a, \FFF_2^{a+b})$ for some nonnegative integers $a, b$. Given any $c \ge \max(u, 0)$, we then have
\begin{align*}
&\int_{d_1(A) = c} \#\Inj(O, A) d\mu(A) \\
&\qquad =\, P^{\textup{Mat}}(c \,|\, (\infty - u) \times \infty) \cdot \#\Inj\left(\FFF_2^{a}, \FFF_2^{c}\right) 2^{bc}\cdot \sum_{d \le c - u} P^{\textup{Alt}}(d \,|\, c - u)  \cdot \#\Inj\left(\FFF_2^{b},\,\FFF_2^{d}\right)
\end{align*}
by \eqref{eq:InjC}. 

Take $T$ to be a random alternating pairing on $\FFF_2^{c-u}$ uniformly selected from all possibilities.   The sum on the right of this equation may be interpreted as the expected number of injections of $\FFF_2^{b}$ into $\FFF_2^{c-u}$ such that the image of the injection lies in the kernel of $T$. But, given any subspace of dimension $b$ of $\FFF_2^{c-u}$, the probability that it lies in the kernel of $T$ is
\[2^{-b(c-u) + \tfrac{1}{2}b(b+1)}.\]
By linearity of expectation, the above integral is then
\[P^{\textup{Mat}}(c\, |\, (\infty - u) \times \infty) \cdot \#\Inj\left(\FFF_2^{a}, \FFF_2^{c}\right) \cdot \#\Inj\left(\FFF_2^{b}, \, \FFF_2^{c - u}\right) 2^{bu + \tfrac{1}{2}b(b+1)}.\]
From the explicit expression fo $P^{\text{Mat}}$ in \cite[Section 2]{Smi22b}, we thus find that the integral of the proposition equals the limit of
\[ 2^{bu + \tfrac{1}{2}b(b+1)} \cdot \sum_{c \ge 0} P^{\textup{Mat}}(c\, |\, (n - u) \times n) \cdot \#\Inj\left(\FFF_2^{a}, \FFF_2^{c}\right) \cdot \#\Inj\left(\FFF_2^{b}, \, \FFF_2^{c - u}\right)\]
as $n$ tends to infinity.

We now will calculate this sum for a sufficiently large integer $n$. Take $T$ to be a random $(n- u) \times n$ matrix over $\FFF_2$ uniformly selected from all possibilities. Then we may interpret the above sum as the expected number of pairs of injections
\[\lambda: \FFF_2^a  \to \FFF_2^n \quad\text{and}\quad \psi: \FFF_2^b \to \FFF_2^{n-u}\]
such that the image of $\lambda$ is in the right kernel of $T$ and the image of $\psi$ is in the left kernel of $T$. Given $\lambda$ and $\psi$, the probability that this condition holds is
\[2^{-a(n - u) - b(n - a)},\]
so this sum is
\[2^{-an + au - bn + ab} \cdot \# \Inj\left(\FFF_2^{a},\, \FFF_2^n\right) \cdot \#\Inj\left(\FFF_2^b,\, \FFF_2^{n-u}\right).\]
As $n$ tends to infinity, this tends to
\[2^{au + ab - bu},\]
and we can conclude that
\[\int \# \Inj(O, A) d\mu(A) = 2^{au + ab + \tfrac{1}{2}b(b+1)},\]
as claimed.
\end{proof}

\begin{prop}
\label{prop:IV_unique}
Take $\mu_1, \mu_2, \dots$ and $\mu$ to be probability measures on $\mathscr{C}_{\cong}$. Suppose that there is some $C > 0$ such that, given any $(V_0, V)$ in $\mathscr{C}_{\cong}$ with $\dim V > 2\dim V_0 + C$, we have $\mu_i(V_0, V)  = 0$ for all $i$. Suppose further that
\[\lim_{i \to \infty} \int \#\Inj\left(O, A\right) d\mu_i(A) =   \int \#\Inj\left(O, A\right) d\mu(A) \le 2^{(a+b)C + ab + \frac{1}{2}b^2}\]
for all objects $O$ in $\mathscr{C}$, where $a = d_1(O)$ and $b = d_2(O)$. Then the $\mu_i$ have limit $\mu$.
\end{prop}
\begin{proof}
Choose a nonnegative integer $b$, and define a measure $\nu$ on the set of isomorphism classes of finite vector spaces over $\FFF_2$ by
\[\nu(\FFF_2^c) = \sum_{d \ge 0} \#\Inj\left(\FFF_2^b, \FFF_2^d\right) 2^{bc} \mu\big(\FFF_2^c, \FFF_2^{c+d}\big).\]
Define $\nu_i$ from $\mu_i$ analogously for each nonnegative integer $i$. Applying \eqref{eq:InjC} gives
\[  \int \#\Inj\left((\FFF_2^a, \FFF_2^{a+b}), A\right) d\mu(A) = \sum_{c \ge 0} \#\Inj\big(\FFF_2^a, \,\FFF_2^c\big) \nu(\FFF_2^c)\] 
for any nonnegative integer $a$. So
\[\lim_{i \to \infty} \sum_{c \ge 0} \#\Inj\big(\FFF_2^a, \,\FFF_2^c\big) \nu_i(\FFF_2^c) \,=\, \sum_{c \ge 0} \#\Inj\big(\FFF_2^a, \,\FFF_2^c\big) \nu(\FFF_2^c).\] 
We claim that applying \cite[Theorem 1.6]{Saw22} to the diamond category of finite $\FFF_2$-vector spaces gives that the $\nu_i(\FFF_2^c)$ limit to $\nu(\FFF_2^c)$. To prove this, we note by \cite[Lemma 3.8]{Saw22} that the associated M\"{o}bius function evaluated at $\FFF_2^a$ and $\FFF_2^c$ is bounded by $2^{\frac{1}{2}(a^2+a)}$ for $a \ge c$, while the number of automorphisms of $\FFF_2^a$ has growth rate on the order of $2^{a^2}$, so the expression before \cite[(1.4)]{Saw22} is bounded for this example.

So, for $b,c \ge 0$,
\[\sum_{d \ge 0} \#\Inj\left(\FFF_2^b, \FFF_2^d\right)  \mu\big(\FFF_2^c, \FFF_2^{c+d}\big) = \lim_{i \to \infty} \sum_{d \ge 0} \#\Inj\left(\FFF_2^b, \FFF_2^d\right)  \mu_i\big(\FFF_2^c, \FFF_2^{c+d}\big).\]
By the proposition's assumptions on the $\mu_i$, these moments are bounded by $2^{b(c+2C)}$. Again applying \cite[Theorem 1.6]{Saw22} to the diamond category of finite $\FFF_2$-vector spaces gives the result.
\end{proof}

\begin{proof}[Proof of Theorems \ref{thm:balanced} and \ref{thm:2IV}]
Take $E$ to be a curve in Case IV,  take $\varphi: E \to E_0$ to be its balanced isogeny, and choose a nonzero integer $d_0$. Given $d \in X_E(d_0, \infty)$, we define an object $A(d)$ in $\mathscr{C}$ by
\[\left(\text{im} (H^1(G_{\QQ}, E[\varphi])) \cap  \Sel^{2^{\infty}} E^d, \,\,\text{im} (H^1(G_{\QQ}, E[2])) \cap  \Sel^{2^{\infty}} E^d\right).\]
Given an object $O = (\FFF_2^a, \FFF_2^{a+b})$ in $\mathscr{C}$, we have
\[\# \mathcal{S}_{a, b}(E^d) = 2^{a} \cdot  \left(\#H^0(G_{\QQ}, E[2])\right)^{b} \cdot \# \Inj(O, A(d))\]
for all but finitely many $d$ in $X_E(d_0, \infty)$. So Theorem \ref{thm:moments} gives
\[\lim_{H \to \infty} \frac{1}{\#X_E(d_0, H)}  \sum_{d \in X_E(d_0, H)} \# \Inj(O, A(d)) = 2^{au + ab + \tfrac{1}{2}b(b+1)},\]
and the results follow for $E$ in Case IV by Propositions \ref{prop:IV_consis} and \ref{prop:IV_unique}.

It remains to prove Theorem \ref{thm:balanced} for $E$ in Case V with a balanced isogeny $\varphi: E \to E_0$. Take $\varphi': E_0 \to E$ to be the dual isogeny. $E_0$ is in Case IV, so the theorem holds for the isogeny $\varphi'$. But the result now follows for $\varphi$ since
\[P^{\text{Mat}}(a\,|\, (\infty - u) \times \infty) = P^{\text{Mat}}(a - u\,|\, (\infty + u) \times \infty)\]
for any integer $u$ and any $a \ge \max(u, 0)$.
\end{proof}

We now define the necessary category to prove Theorem \ref{thm:2V_II}.
\begin{defn}
Fix a finite $\FFF_2$-vector space $\mathcal{L}$, and take $\mathscr{D}$ to be the category of homomorphisms $V \to \mathcal{L}$, where $V$ is a finite $\FFF_2$-vector space. The morphisms from this object to one of the form $W \to \mathcal{L}$ will be the homomorphisms $V \to W$ such that the corresponding triangle is commutative. We take $\mathscr{D}_{\cong}$ to be the set of isomorphism classes of objects in this category, and we endow this set with the discrete topology. Given objects $O, A$ in this category, we write $\text{Inj}(O, A)$ for the set of monomorphisms from $O$ to $A$.

Given an object $O = (V \to \mathcal{L}) $ in this category, we define $d(O) = \dim V$ and $\lambda(O)$ to be the image of $V$ in $\mathcal{L}$.
\end{defn}
\begin{prop}
\label{prop:V_consis}
Fix a nonzero integer $u$, and define a measure $\mu$ on $\mathscr{D}_{\cong}$ by
\[\mu(A) = P^{\textup{Mat}}(d(A)\, |\, (\infty - u) \times \infty) \cdot P_{\textup{loc}}\big(\FFF_2^{d(A)}, \lambda(A)\big),\]
where $P_{\textup{loc}}\big(\FFF_2^{d(A)}, V\big)$ denotes the probability that a homomorphism $\FFF_2^{d(A)} \to \mathcal{L}$ selected uniformly at random has image $V$.

Then, given an object $O$ in $\mathscr{D}$,
\[\int \# \Inj(O, A) d\mu(A) = (\# \mathcal{L})^{-d(O)}  2^{u\cdot d(O)}.\]
\end{prop}
\begin{proof}
The integral of the proposition equals
\[\sum_{c \ge 0} P^{\textup{Mat}}(c\, |\, (\infty - u) \times \infty) \cdot \sum_{\substack{V \subseteq \mathcal{L} \\  \dim V \le c}} P_{\textup{loc}}(\FFF_2^c, V) \cdot \#\Inj\left(O,\, A(c, V)\right),\]
where $A(c, V)$ is any object $A$ with $d(A) = c$ and $\lambda(A) = V$.

Take $L: \FFF_2^c \to \mathcal{L}$  to be a random homomorphism selected uniformly at random from all possibilities. The inner sum above may be interpreted as the expected number of injections from $\FFF_2^{d(O)}$ to $\FFF_2^c$ such that the homomorphism of $O$ equals the composition 
\[\FFF_2^{d(O)} \to \FFF_2^c  \xrightarrow{\,\,L\,\,} \mathcal{L}.\]
Given an injection from $\FFF_2^{d(O)}$ to $\FFF_2^c$, the probability this composition agrees with the homomorphism of $O$ is $2^{-d(O) \cdot \dim \mathcal{L}}$. So the integral of the proposition is
\[2^{-d(O) \cdot \dim \mathcal{L}} \cdot \sum_{c \ge 0} P^{\textup{Mat}}(c\, |\, (\infty - u) \times \infty) \cdot \# \Inj\big(\FFF_2^{d(O)}, \,\FFF_2^c\big),\]
and the end of the proof follows that of Proposition \ref{prop:IV_consis}.
\end{proof}

\begin{prop}
\label{prop:V_uniqueness}
Take $\mu_1, \mu_2, \dots$ and $\mu$ to be positive measures on $\mathscr{D}_{\cong}$. Suppose there is $C > 0$ such that
\[\lim_{i \to \infty} \int \#\Inj(O, A) d\mu_i(A) =  \int \# \Inj(O, A) d\mu < 2^{Cd(O)}\]
for all objects $O$ in $\mathscr{D}$. Then the $\mu_i$ have limit $\mu$.
\end{prop}
\begin{proof}
The category $\mathscr{D}$ is the dual of a diamond category as defined in \cite[Definition 1.3]{Saw22}, and the result follows immediately from \cite[Theorem 1.6]{Saw22}.
\end{proof}

\begin{proof}[Proof of Theorem \ref{thm:2V_II}]
Take all notation as in the theorem statement. Define the category $\mathscr{D}$ with respect to $\mathcal{L}$. Given $d \in X_E(d_0, H)$, we take $A(d)$ to be the object in $\mathscr{D}$ given by
\[\text{loc}: \left(\text{im}\left(H^1(G_{\QQ}, E^d[2])\right) \cap \Sel^{2^{\infty}} E^d \right) \to \mathcal{L}.\]
Given $O = \lambda: \FFF_2^a \to \mathcal{L}$, define $v(O)$ to be the element in $\mathcal{L}^{\oplus a}$ whose $i^{th}$ comoponent equals the image of $\lambda$ on the basis vector $e_i$ for $i \le a$. Then
\[\# \mathcal{S}_{a, \langle v(O)\rangle}(E^d) \,-\, \big( |\langle v(O)\rangle | - 1\big) \cdot \# \mathcal{S}_{a, 0}(E^d)   \,=\, 2^a \cdot \#\Inj(O, A(d)) \]
for all but finitely many $d$ in $X_E(d_0, \infty)$. So Theorem \ref{thm:moments} gives
\[\lim_{H \to \infty} \frac{1}{\#X_E(d_0, H)}  \sum_{d \in X_E(d_0, H)} \# \Inj(O, A(d)) = (\# \mathcal{L})^{-a} \cdot 2^{au},\]
and the result follows from Propositions \ref{prop:V_consis} and \ref{prop:V_uniqueness}.
\end{proof}

\section{The distribution of $2^{\infty}$-Selmer coranks}
\label{sec:2infty}

The goal of this section is to prove Theorem \ref{thm:IVV_pretrick}, which together with the work in Section \ref{ssec:isogeny_tricks} will show that the twist families of curves in Cases IV and V have the expected distribution of $2^{\infty}$-Selmer coranks. So suppose this theorem did not hold for a given curve $E$ in Case IV or V.

If $E$ is in Case IV, take $\varphi$ to be its balanced isogeny. If $E$ is in Case V, take $\varphi_1, \varphi_2$ to be its pair of balanced isogenies. We have then assumed that
\begin{equation}
\label{eq:central_absurd}
r_{2^{\infty}}(E^d) >\begin{cases} \max\big(1, r_{\varphi}(E^d) \big) &\text{in Case IV }\\ \max\big(1, r_{\varphi_1}(E^d), r_{\varphi_2}(E^d) \big) &\text{in Case V}\end{cases}
\end{equation}
holds for a positive proportion of squarefree integers $d$.

Recall that we defined good grids of ideals associated to $E$ and their associated twists in Notation \ref{notat:grid}.  From \cite[Proposition 8.5]{Smi22b}, we conclude that there is $\delta > 0$ so that, for any $H_0 > 0$, there is a good grid of ideals $X$ of height $H > H_0$ indexed by a set $S$ and an associated choice of function $\chi: X \to \Z^{\ne 0}$ such that the proportion of $x$ in $X$ such that $d = \chi(x)$ satisfies \eqref{eq:central_absurd} is at least $\delta$.

So long as $H_0$ is sufficiently large, the grid $X$ satisfies conditions (1) and (2) of \cite[Definition 4.19]{Smi22a} by the definition of a good grid of ideals and the Chebotarev density theorem. 

If we choose $x_0$ uniformly at random from $X$, the probability that the grid class $\class{x_0}$ does not satisfy condition (3) of this definition tends to $0$ as $H$ grows. For the first part of the condition, this follows since the number of grid classes can be bounded by an expression of the form $e^{C|S|^2}$, where $C > 0$ does not depend on $H_0$. For the second part, it follows from Proposition \ref{prop:moments}.

Again applying Proposition \ref{prop:moments}, we find there is some $R > 0$ such that the probability that 
\begin{equation}
\label{eq:abs_2_bnd}
r_2(E^{\chi(x_0)}) \le R
\end{equation}
is at least $1 - \tfrac{1}{2} \delta$ so long as $H$ is sufficiently large. Fix such an $R$.

Take $E_0 = E^{\chi(x_0)}$, take $r_2 = r_2(E_0)$, and define $r_{\varphi}, r_{\varphi_1}, r_{\varphi_2}$ analogously. Choose a tuple
\[\Phi = (\phi_1, \dots, \phi_{r_2}) \in \begin{cases} \left(\Sel^{\varphi} E_0\right)^{r_\varphi} \oplus  \left(\Sel^{2} E_0\right)^{r_2 - r_{\varphi}} &\text{in Case IV}\\  \left(\Sel^{\varphi_1} E_0\right)^{r_{\varphi_1}} \oplus  \left(\Sel^{\varphi_2} E_0\right)^{r_{\varphi_2}} \oplus \left(\Sel^2 E_0\right)^{r_{2} - r_{\varphi_1} - r_{\varphi_2}}& \text{in Case V}\end{cases}\]
giving a basis for $\Sel^2 E_0/ \text{im}\left(H^0(G_{\QQ}, E[2])\right)$.  The corresponding ramification subspace in $E[2]^{\oplus r_2}$ is cofavored outside of a negligible set of $x_0$ by \cite[Section 10.4]{Smi22b}. By Propositions \ref{prop:IV_cof} and \ref{prop:V_cof}, it follows that it equals 
\begin{alignat*}{2}
&E[\varphi]^{r_{\varphi}} \oplus E[2]^{r_2 - r_{\varphi}}&&\quad\text{in Case IV and}\\
&E[\varphi_1]^{r_{\varphi_1}} \oplus E[\varphi_2]^{r_{\varphi_2}} \oplus E[2]^{r_{2} - r_{\varphi_1} - r_{\varphi_2}}&&\quad\text{in Case V}.
\end{alignat*}
Call this space $Q$.

Define the subsets $S_{\text{pot-pre}}(x_0), S_{\text{pot-a/b}}(x_0) \subseteq S$ as in \cite[Definition 4.19]{Smi22a}. Consider the condition
\begin{align}
\label{eq:defici}
&\text{For } S_0 \text{ equal to either } S_{\text{pot-pre}}(x_0) \text{ or } S_{\text{pot-a/b}}(x_0) \text{ and every }v \in Q,\\
&\nonumber \text{there are at least }  \log^{(3)} H \text{ choices of } s \in S_0 \text{ satisfying } \pi_s(\Phi) = v.
\end{align}
By applying Hoeffding's inequality as in \cite[Section 10.4]{Smi22b}, we find that this condition holds for all but a negligible set of $x_0$. So, if $H$ is sufficiently large, we find that there is a choice of $x_0$ in $X$ such that $\class{x_0}$ satisfies conditions (1), (2), and (3) of \cite[Definition 4.19]{Smi22a} and \eqref{eq:abs_2_bnd}, such that \eqref{eq:defici} holds, and such that \eqref{eq:central_absurd} is satisfied on a subset of $\class{x_0}$ of density at least $\tfrac{1}{4}\delta$.

Define $V_2(x_0)$ and the parameterizing maps $\Psi_x: V_2(x_0) \to \Sel^{2} E^{\chi(x)}$ for $x \in X$ as in Section \ref{ssec:general_Smi1}. By our assumptions, we may choose $v_1, v_2 \in V_2(x_0)$ parameterizing linearly independent elements  $\phi_1, \phi_2$ of $\Sel^2 E_0/ \text{im}\left(H^0(G_{\QQ}, E[2])\right)$ such that $\phi_1$ and $\phi_2$ are  divisible elements in  the $2^{\infty}$-Selmer group of $E^{\chi(x)}$ for at least $4^{-R - 1} \delta \cdot \# \class{x_0}$ choices of $x$ in $\class{x_0}$ and such that $\langle \phi_1, \phi_2 \rangle$ is not contained in the image of the Selmer group associated to any balanced isogeny from $E_0$.

In Case IV, after rechoosing our basis if necessary, we may assume that $v_1$ parameterizes $\phi_{r_{\varphi} + 1}$ and $v_2$ parameterizes either $\phi_1$ or $\phi_{r_{\varphi} + 2}$. There is then $v \in Q$ that is a generator for $E[2]/E[\varphi]$ at coordinate $r_{\varphi} + 1$, equal to the generator for $E[\varphi]$ at the coordinate corresponding to $v_2$, and equal to $0$ at all other coordinates. Since we have assumed the condition \eqref{eq:defici}, $v_1 \otimes v_2$ is a controllable tensor, in the sense of Definition \ref{defn:half_higher}.

In Case V, after rechoosing our basis and swapping the balanced isogenies if necessary, we may assume that $v_1, v_2$ parameterizes one of
\[\left(\phi_{1}, \,\phi_{r_{\varphi_1} +1 }\right)\quad\text{or}\quad \left(\phi_{1}, \,\phi_{r_{\varphi_1} + r_{\varphi_2} +1 }\right)\quad\text{or}\quad \left(\phi_{r_{\varphi_1} + r_{\varphi_2} + 1}, \,\phi_{r_{\varphi_1} + r_{\varphi_2} +2 }\right).\] 
In each of these three situations, we may apply \eqref{eq:defici} as before to  show that $v_1 \otimes v_2$ is a controllable tensor.

In all cases, Theorem \ref{thm:controllable} applies to this tensor. For $k\ge 1$, take $Y_k$ to be the union of higher grid classes $\class{x_1}_k$ contained in $\class{x_0}$ such that $V_{2^k}(x_1)$ contains $v_1$ and $v_2$. Then
\[\sum_{x \in Y_k} \exp\left(2\pi i \cdot \text{ct}_{x, k}(v_1 \otimes v_2) \right) \ge 2|Y_{k+1}| - |Y_k|.\]
So, choosing $M$ to be the least integer such that $2^M \cdot 4^{-R - 1} \delta > 2$, we find that 
\begin{equation}
\label{eq:oops}
\sum_{k = 1}^M \sum_{x \in Y_k} 2^{k-1}\exp\left(2\pi i \cdot \text{ct}_{x, k}(v_1 \otimes v_2)\right) \ge \# \class{x_0}.
\end{equation}
But each $Y_k$ is the disjoint union of a collection of higher grid classes whose number may be bounded only in terms of $R$ and $k$, so the left-hand side of this inequality is bounded by an expression of the form $\mathcal{O}\left( (\log \log H)^{-1/4} \cdot \# \class{x_0}\right)$ by Theorem \ref{thm:controllable}, where the implicit constant does not depend on $H$. This contradicts \eqref{eq:oops} for sufficiently large $H$. \qed


\bibliography{references}{}
\bibliographystyle{amsplain}

\end{document}